\documentclass[a4paper]{article}
\usepackage[paperwidth=220mm]{geometry}

\usepackage[affil-it]{authblk}

\usepackage{amsmath, amsthm, amssymb}

\usepackage{marginnote}
\usepackage{graphicx}
\usepackage{caption}
\usepackage{subcaption}

\usepackage[usenames,dvipsnames,table]{xcolor}

\usepackage{algorithmic}
\usepackage{algorithm}

\usepackage{pifont} 

\usepackage{url}

\usepackage{soul}

\newtheorem{definition}{Definition}

\newtheorem{myth}{Theorem}
\newtheorem{myco}{Corollary}

\newcommand{\Probab}[1]{\mathcal{P}({#1})}
\newcommand{\Pcond}[2]{\Probab{{#1}\mid{#2}}}
\newcommand{\Pconj}[2]{\Probab{{#1} \wedge {#2}}}
\renewcommand{\~}[1]{\overline{#1}}

\newcommand{\ALGNAME}{CAPRI}

\newcommand{\ALGNAMEPLOS}{CAPRESE}

\newcommand{\causes}{\rhd}
\renewcommand{\models}{\Vvdash}

\newcommand{\PR}{\textsc{pr}}
\newcommand{\TP}{\textsc{tp}}

\newcommand{\at}{atomic}

\newcommand{\cc}{\boldsymbol{c}}

\newcommand{\conj}[1]{\text{clauses}\left({#1}\right)}

\newcommand{\setbp}{\textsc{setbp1}}
\newcommand{\asxl}{\textsc{asxl1}}
\newcommand{\tet}{\textsc{tet2}}

\newcommand{\aCML}{\textsc{aCML}}
\newcommand{\cbl}{\textsc{cbl}}
\newcommand{\sfb}{\textsc{sf3b1}}
\newcommand{\ezh}{\textsc{ezh2}}
\newcommand{\idh}{\textsc{idh2}}

\newcommand{\nomark}{\color{red}{\ding{55}}}
\newcommand{\okmark}{\color{green}{\ding{51}}}

\newcommand{\Newrev}[1]{{\color{black}{#1}}}

\title{CAPRI: Efficient Inference of Cancer Progression
 Models from Cross-sectional Data}
 
\author[1]{Daniele Ramazzotti\thanks{To whom correspondence should be addressed.}}
\author[1]{Giulio Caravagna}
\author[2]{Loes Olde Loohuis}
\author[1]{Alex Graudenzi}
\author[3]{Ilya Korsunsky}
\author[1]{Giancarlo Mauri}
\author[1]{Marco Antoniotti}
\author[3]{Bud Mishra}
\affil[1]{Dept. of Informatics, Systems and Communication, University of Milan Bicocca, Milan, Italy}
\affil[2]{Center for Neurobehavioral Genetics, University of California Los Angeles, Los Angeles, USA}
\affil[3]{Courant Institute of Mathematical Sciences, New York University, New York, USA}

\date{}

\begin{document}



%

\maketitle

\vspace{-0.5mm}


\begin{abstract}

\noindent
We devise a novel inference 
 algorithm to effectively solve the \emph{cancer
  progression model reconstruction} problem.  Our empirical analysis
of the accuracy and convergence rate of our algorithm,
\emph{CAncer PRogression Inference} (CAPRI), 
shows that it outperforms the state-of-the-art algorithms 
addressing similar problems.

\paragraph{\textbf{Motivation:}}
 Several cancer-related genomic data have become available (e.g., \emph{The Cancer
  Genome Atlas}, TCGA) typically involving hundreds of patients.  
 At present, most of these data are aggregated in a
\emph{cross-sectional} fashion providing all measurements at the time of diagnosis.

Our goal is to infer cancer ``progression'' models from such 
data. These models are represented as \Newrev{directed acyclic graphs (DAGs)} of 
collections of ``\emph{selectivity}'' relations,
where a mutation in a gene
\textit{A} ``selects'' for a later mutation in a gene \textit{B}.
Gaining insight into the structure of such
progressions has the potential to improve both the stratification of patients
and personalized therapy choices.

\paragraph{\textbf{Results:}}
The CAPRI algorithm relies on a scoring method based on a {\em probabilistic theory\/} developed by Suppes,
coupled with \emph{bootstrap}
and \emph{maximum likelihood} inference. The resulting algorithm
is efficient, achieves high accuracy, 
and has 
good complexity, also, in terms of convergence properties.  
CAPRI performs especially well
 in the presence of noise in the data, and with limited
sample sizes.  Moreover CAPRI, in contrast to other approaches, 
robustly reconstructs different types of confluent trajectories despite
irregularities in the data.

We also report on an ongoing investigation using
CAPRI to study  \emph{atypical Chronic Myeloid Leukemia},  \Newrev{in which we uncovered non trivial selectivity relations and exclusivity patterns among key genomic events. }


\paragraph{\textbf{Availability:}}
CAPRI is part of the \emph{TRanslational ONCOlogy} R package and is
freely available on the web at:\\
\texttt{http://bimib.disco.unimib.it/index.php/Tronco}

\paragraph{\textbf{Contact:}}
\href{daniele.ramazzotti@disco.unimib.it}{daniele.ramazzotti@disco.unimib.it}
\end{abstract}




\section{Introduction}
\label{sect:introduction}
Analysis and interpretation of the fast-growing biological data sets
that are currently being curated from laboratories all over the world
require sophisticated computational and statistical methods. 

Motivated by the availability of genetic patient data,
we focus on the problem of
\emph{reconstructing progression models} of cancer. In particular, we
aim to infer
the plausible sequences of \emph{genomic alterations} that, by a
process of \emph{accumulation}, selectively make a tumor fitter to
survive, expand and diffuse (i.e., metastasize). 
Along the trajectories of progression, a tumor (monotonically) acquires or
``activates'' mutations in the genome, which, in turn, produce progressively more
``viable'' clonal subpopulations over the so-called 
\emph{cancer evolutionary landscape}
(cfr., \cite{Merlo:2006em,huang_2009,Vogelstein29032013}). 

 Knowledge of such
progression models is very important for drug development and in
therapeutic decisions. For example, it has been known that for the same cancer type, patients
in different stages of different progressions  respond differently to different 
treatments.


Several datasets are currently available that aggregate 
diverse cancer-patient data and report in-depth mutational profiles, 
including e.g., structural changes (e.g., inversions, translocations, copy-number variations) or somatic mutations
(e.g., point mutations, insertions, deletions, etc.). An example of
such a dataset is \emph{The Cancer Genome Atlas} (TCGA)
(cfr., \cite{tcga})).  These data, by their 
very nature,
only give a snapshot of a given tumor sample,
mostly from biopsies of untreated tumor samples at the time of diagnoses. It still remains
impractical to track the tumor progression in any single patient over
time, thus limiting most analysis methods to work with
\emph{cross-sectional} data\footnote{Unlike longitudinal
  studies, these cross-sectional data are derived from samples that are collected at
  unknown time points, and can be considered as ``static''.}.


To rephrase, we focus on the problem of \emph{cancer
  progression models reconstruction from cross-sectional
  data}.  The problem is not new and, to the best of our knowledge,
two threads of research starting in the
late 90's have addressed it.
The first category of works examined mostly gene-expression
data to reconstruct the temporal ordering of samples
(cfr., \cite{magwene03:_reconstructing,gupta_bar_joseph08:_extracting}).
The second category of works 
looked at inferring cancer
progression models of increasing model-complexity, starting from the simplest tree models
(cfr. \cite{desper_1999}) to more complex graph models (cfr.,
\cite{gerstung2009quantifying}); see the next subsection for an overview of the state of
the art.  
%
%
Building on our previous work described in \cite{Loohuis:2014im} we present a novel
and comprehensive algorithm of the second category that addresses this
problem.

\begin{figure*}[t]
  \begin{center}
    \includegraphics[width=0.80\textwidth]{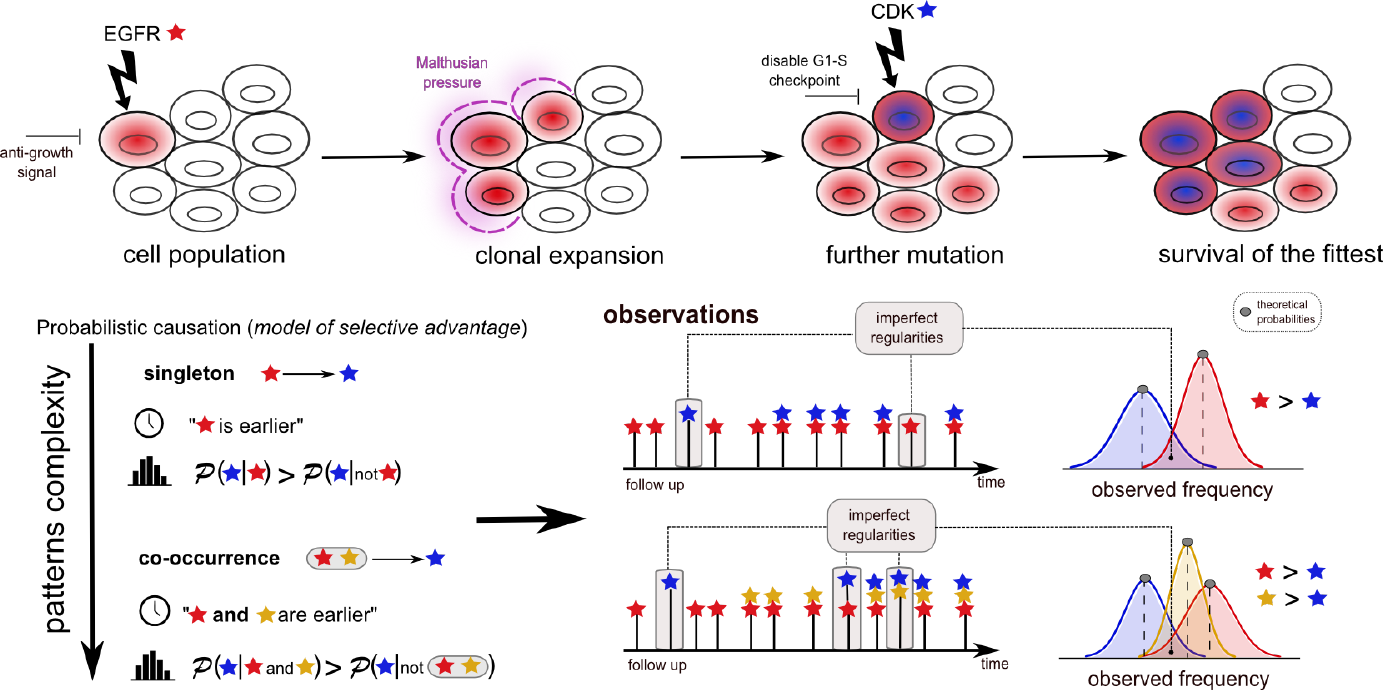}
  \end{center}
  \vspace*{.05in}
  \caption[Selectivity Relation in Tumor Evolution]{\textbf{Selectivity Relation in Tumor Evolution.}
  The \emph{CAncer PRogression Inference} (CAPRI) algorithm
    examines cancer patients' genomic cross-sectional data to determine 
    relationships among genomic alterations (e.g., somatic mutations,
    copy-number variations, etc.) that modulate the somatic evolution
    of a tumor. When CAPRI concludes that aberration $a$ (say,  an
    \textsc{egfr} mutation) ``selects for'' aberration $b$ (say, a
    \textsc{cdk} mutation),
    such relations can be rigorously expressed using Suppes'
    conditions, which postulates that if $a$ selects $b$,
    then $a$ occurs before $b$ (\emph{temporal priority}) and occurrences of
    $a$ raises the probability of emergence of $b$ (\emph{probability
    raising}).  Moreover, CAPRI is capable of reconstructing relations
    among more complex boolean combination of events, as shown in the
    bottom panel and discussed in the Approach section.
  }
\label{fig:overview}
\end{figure*}

The new algorithm proposed here is called \emph{CAncer PRogression
  Inference} (CAPRI) and is part of the \emph{TRanslational ONCOlogy}
(TRONCO) package (cfr., \cite{tronco}). Starting from cross-sectional
genomic data, CAPRI reconstructs a probabilistic progression model 
by inferring ``selectivity relations'', where a mutation in a gene
\textit{A} ``selects'' for a later mutation in a gene \textit{B}. 
These relations are depicted in a combinatorial graph and resemble the way a mutation exploits its
 ``\emph{selective advantage}''
to allow its host cells to expand
 clonally.  Among other things, a selectivity relation implies a putatively invariant temporal structure among the
genomic alterations (i.e., \emph{events}) in a specific cancer type.
In addition, these relations are expected to also imply ``probability raising'' for a pair of events in the following sense:
Namely, a selectivity relation between a pair of events here signifies that
the presence of the earlier genomic alteration  (i.e., the
\emph{upstream event}) that is advantageous in a
Darwinian competition scenario
increases the probability with which a subsequent advantageous
genomic alteration (i.e., the \emph{downstream event})
appears in the clonal evolution of the tumor. Thus the selectivity relation captures the effects of the
evolutionary processes, and not just correlations among the events and imputed clocks associated with them. As an example, we show in (Figure~\ref{fig:overview}) the selectivity relation connecting a mutation of \textsc{egfr} to the mutation of \textsc{cdk}.

Consequently, an inferred selectivity relation suggests
mutational profiles in which certain samples (early-stage patients)
 display specific alterations only (e.g., the alteration
characterizing the beginning of the progression), while certain other
samples (e.g., late-stage patients) display a superset subsuming the early
mutations (as well as alterations that occur subsequently in the progression).

Various kinds of genomic aberrations are suitable as input data, and
include somatic point/indel mutations, copy-number alterations, etc., provided that
they are \emph{persistent}, i.e., once an alteration is acquired
no other genomic event can restore the cell to the non-mutated
(i.e., \emph{wild type}) condition\footnote{For instance,
  epigenetic alterations such as methylation and  alterations in gene
  expression are not directly usable as input data  for the
  algorithm.  Notice that the selection of the relevant events  is
  beyond the scope of this work and requires a further upstream
  pipeline, such as that provided, for instance,
  in \cite{Tamborero:2013nx,Vogelstein29032013}.}.

The selectivity \Newrev{relations} that CAPRI
reconstructs are ranked and subsequently further refined by means of a
hybrid algorithm, \Newrev{which reasons over time, mechanism and chance, as follows.}
CAPRI's overall scoring methods combine topological constraints
grounded on Patrick Suppes' conditions of probabilistic causation (see
e.g., \cite{Suppes70}), with a \emph{maximum likelihood-fit} procedure
(cfr., \cite{bayesian_learning}) and derives much of its statistical
power from the application of \emph{bootstrap} procedures (see e.g.,
\cite{efron_1982}).
CAPRI returns a \emph{graphical model} of a complex selectivity
relation among events which captures the essential aspects of cancer
evolution: branches, confluences and independent progressions. In the
specific case of confluences, CAPRI's ability to infer them is related
to the complexity of the ``patterns'' they exhibit, expressed in a
logical fashion.  As pointed out by other approaches (cfr.,
\cite{beerenwinkel_2007}), this strategy requires trading off
complexity for expressivity of the inferred models, and results in two
execution modes for the algorithm: supervised and unsupervised, which
we discuss in details in Sections~\ref{sect:approach}
and~\ref{sect:methods}.

In Section~\ref{sect:methods} (Methods) we show that CAPRI enjoys a
set of attractive properties in terms of its complexity, soundness and
expressivity, even in the presence of uniform \emph{noise} in the input
data -- e.g., due to ~\emph{genetic heterogeneity} and experimental errors.
Although many other approaches enjoy similar asymptotic properties, we
show that CAPRI can compute accurate results with surprisingly
small sample sizes (cfr., Section~\ref{sect:discussion}).  Moreover, to the best of
our knowledge, based on extensive synthetic data simulations, CAPRI
outperforms all the competing procedures with respect to all desirable 
performance metrics. We conclude by showing an application of CAPRI to
reconstruct a progression model for \emph{atypical Chronic Myeloid
  Leukemia} (aCML) using a recent exome sequencing dataset, first presented in
\cite{piazza_nat_gen}.

\subsection{State of the Art}
\label{sect:state-of-the-art}

For an extensive review on \emph{cancer progression model
  reconstruction} we refer to the recent survey by \cite{Beerenwinkel:2014eb}.
In brief, progression models for cancer have been studied starting with the
seminal work of \cite{vogelstein1988genetic} where, for the first
time, cancer progression was described in terms of a directed path by assuming
the existence of a unique and most likely temporal order of genetic mutations.
\cite{vogelstein1988genetic} manually created a (colorectal) cancer
progression from a genetic and clinical point of view. More rigorous and complex
algorithmic and statistical automated approaches have appeared subsequently.
As stated already, the earliest thread of research simply sought more generic
progression models that could assume tree-like structures.
%
%
The \emph{oncogenetic tree model} captured evolutionary branches of
mutations (cfr., \cite{desper_1999,szabo}) by optimizing a
\emph{correlation}-based score.  Another popular approach to reconstruct
tree structures appears in \cite{desper_2000}. Other general Markov chain models such as, e.g., \cite{hjelm2006new} reconstruct more flexible probabilistic networks, despite a computationally expensive parameter estimation.
In \cite{Loohuis:2014im},
we introduced an algorithm called \emph{CAncer
  PRogression Extraction with Single Edges} (CAPRESE), which, based on its extensive empirical analysis, may be deemed as the
current state-of-the-art algorithm for the inference of tree models of
cancer progression. It is based on a shrinkage-like statistical estimation, grounded
in a general theoretical framework, which we extend further in this paper.  Other
results that extend tree representations of cancer evolution
exploit mixture tree models, i.e., multiple oncogenetic trees, each of
which can independently result in cancer development (cfr.,
\cite{beerenwinkel_2005}).
In general, all these methods are capable of modeling diverging
temporal orderings  of events in terms of branches, although the
possibility of converging evolutionary  paths is precluded. 

To overcome this limitation, the most recent approaches tends to adopt
Bayesian graphical models, i.e., Bayesian Networks (BN).  In the
literature, there have been two initial families of methods aimed at
inferring the structure of a BN from data (cfr.,
\cite{bayesian_learning}).  The first class of models seeks to explicitly capture
all the conditional independence relations encoded in the edges and
will be referred to as \emph{structural approaches}; the methods in
this family are inspired by the work on causal theories by Judea Pearl
(cfr.,
\cite{pearl1988probabilistic,pearl,spirtes2000causation,tsamardinos2003algorithms}).
The second class -- \emph{likelihood approaches} -- seeks a
model that maximizes the likelihood of the data (cfr.,
\cite{bic_1978,bde_1995,carvalho2009scoring}).

A more recent \emph{hybrid approach} to learn a BN  
which combines the two families above by $(i)$ constraining the
search space of the valid solutions  and, then, $(ii)$ fitting the
model with likelihood maximization (see \cite{beerenwinkel_2007,gerstung2009quantifying,misra2014inferring}).
A further technique to reconstruct progression models from cross-sectional data was introduced in \cite{attolini2010mathematical}, in which the transition probabilities between genotypes are inferred by defining a Moran process that describes the evolutionary dynamics of mutation accumulation. In \cite{cheng2012mathematical} this methodology was extended to account for pathway-based phenotypic alterations.





\section{Approach}
\label{sect:approach}

In what follows, we  denote with $\Probab{\cdot}$ and
$\Pcond{\cdot}{\cdot}$ the observed  marginal and conditional
probability of an event, whose complement is denoted with the
diacritical mark  $\~\cdot$ (macron).

\paragraph{\bfseries A probabilistic model of selective advantage.}
Central to 
CAPRI's score function is Suppes' notion of
\emph{probabilistic causation} (cfr., \cite{Suppes70}), 
which can be stated in
the following terms: a selectivity relation\footnote{Suppes presents the relation in terms of causality; however, we avoid Suppes' terminology as we build on just two of his many axioms, which only give rise to the notion of {\em prima-facie\/} causality.}
among two observables $i$ and $j$ if
(1) $i$ occurs earlier than $j$ -- \emph{temporal priority} (\TP) --
and (2) if the probability of observing $i$ raises the probability of
observing $j$, i.e., $\Pcond{j}{i} > \Pcond{j}{\~i}$ --
\emph{probability raising} (\PR).
The definition of probability raising subsumes positive
statistical dependency and mutuality (see, e.g., \cite{Loohuis:2014im}).
Note that the resulting relation (also, called prima facie causality) is purely
observational and remains agnostic to the possible mechanistic
cause-effect relation involving $i$ and $j$.

While Suppes' definition of probabilistic causation has known limitations in the context of general
causality theory (see discussions in, e.g.,
\cite{probabilistic-causation,kleinberg-thesis}), 
in the context of cancer evolution, this relation appropriately describes various features of
\emph{selective advantage} in somatic alterations that accumulate
as tumor progresses. 

Thus, in our framework,  we implement the temporal priority among
events -- condition (1) -- as $\Probab{i} > \Probab{j}$, because it is
intuitively sound to assume that the (cumulative) genomic events
occurring earlier are the ones present in higher frequency in a
dataset. 
In addition, condition (2) is implemented as is\Newrev{, that is by requiring that for each pair of 
observables $i$ and $j$ directly connected, $\Pcond{j}{i} >
\Pcond{j}{\~i}$ is verified}.
Taken together, these conditions 
gives rise to a natural ordering relation among events, written ``$i
\causes j$'' and read as ``$i$ has a selective influence on $j$.''  
This relation is a \emph{necessary} but \emph{not
  sufficient} condition to capture the notion of selective advantage,  and additional
constraints need to be imposed to filter spurious relations. 
Spurious correlations are both intrinsic to the
definition (e.g., if $i \causes j \causes w$ then also $i \causes w$,
which could be spurious) and to the model we aim at inferring, because
data is finite as well as corrupted by noise.

Building on this framework, we
devise inference algorithms that capture the essential aspects of heterogeneous
cancer progressions: \emph{branching}, \emph{independence} and
\emph{convergence} -- all combining in a progression model.

\begin{figure}[t]
  \begin{center}
    \includegraphics[width=0.35\textwidth]{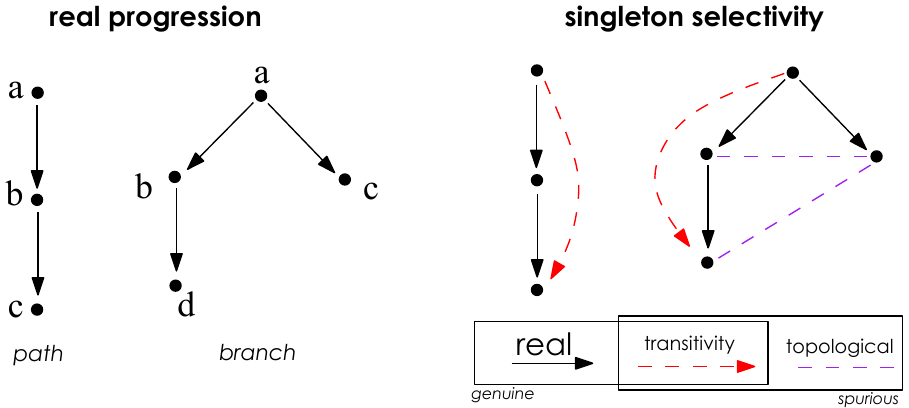} 
  \end{center}
  \begin{center}
    \tiny
    \begin{tabular}{ l ||  l }
      {\ } & \emph{temporal priority}  \\ \hline
      path &  $\Probab{a} > \Probab{b} > \Probab{c}$  \\
      branch & $\Probab{a} > \Probab{b} > \Probab{d}$
               and $\Probab{a} > \Probab{c}$    
    \end{tabular}
  \end{center}
  $\;$ \\
  \begin{center}
    \includegraphics[width=0.35\textwidth]{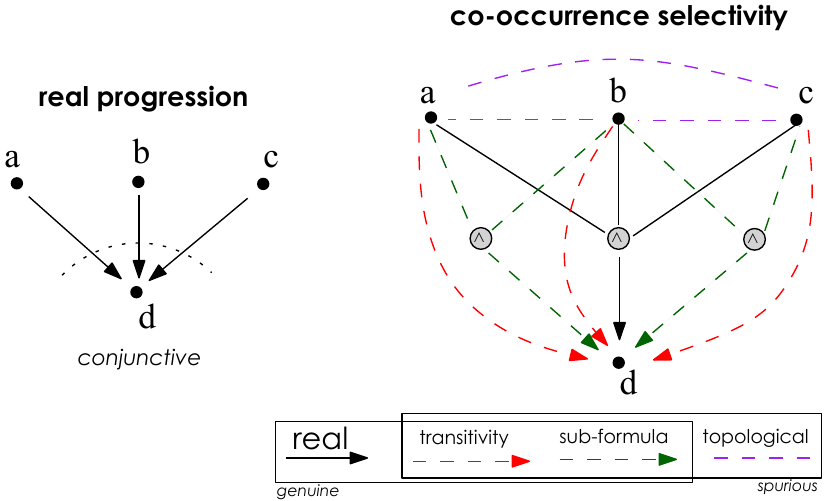} 
  \end{center}
  \begin{center}

    \tiny 
    \begin{tabular}{ l ||  l }
      {\ } & \emph{temporal priority}  \\ \hline
      co-occurrence & $\Probab{a} > \Probab{d}$
                      and $\Probab{b} > \Probab{d}$
                      and $\Probab{c} > \Probab{d}$     
    \end{tabular}
  \end{center}
    %
  \caption[Singleton and co-occurrence selectivity patterns.]{\textbf{Singleton and Co-occurrence Selectivity Patterns.}
   Examples of patterns that CAPRI can automatically extract without prior hypotheses.
    \emph{(Top):} A linear path and branching model (left) and
    corresponding singleton selectivity patterns with infinite sample
    size (right). All the genuine connections are shown (red and black, directed
    by the temporal priority), as well as edges
    (purple, undirected) which might be suggested by the topology (or
    observations, if data were finite).
    \emph{(Bottom):} Example of conjunctive model ($a$ \emph{and} $b$
    \emph{and} $c$). The co-occurrence selectivity pattern is shown, with
    all true patterns and infinite sample size. The topology is augmented
    by logical connectives; green arrows are spurious patterns emerging from
    the structure of the true pattern $a \wedge b \wedge c \causes d$.
      }
  \label{fig:single-cause}
\end{figure}

\paragraph{\textbf{Progression patterns.}}
The complexity of cancer requires modeling multiple
non-trivial \emph{patterns} of its
progression: for a specific event, a pattern is defined as a specific
combination of the closest upstream events that confers a selective
advantage.

 As an example, imagine a clonal subpopulation becoming fit -- thus
 enjoying expansion and selection -- once it acquires a mutation of
 gene $c$, provided it also has previously acquired a
mutation in a gene in the upstream
$a$/$b$ pathway. In terms of progression, we would like to capture the
 trajectories: $\{a,\neg b\}, \{\neg a, b\}$ and $\{a,b\}$ precedes
 $c$ (where $\neg$ denotes the absence of an event in the gene).

To establish this analysis formally, we augment
our
model of selection in a tumor with a language built from simple
propositional logic formulas using the usual Boolean connectives: namely,
``and'' ($\wedge$), ``or'' ($\vee$) and ``xor''
($\oplus$).  These patterns can be
described by formul\ae\ in a propositional logical language, which can be
rendered in \emph{Conjunctive Normal Form} (CNF). A CNF formula
$\varphi$ has the following syntax: $\varphi = \cc_1 \wedge \ldots
\wedge \cc_n$, where each $\cc_i$ is a \emph{disjunctive clause}
$\cc_i = c_{i,1} \vee \ldots \vee c_{i,k}$ over a set of literals,
each literal representing an event or its negation.
Given this (rather obvious) pattern representation,
we write the conditions for 
\emph{selectivity with patterns} as 
\begin{equation}
  \varphi \causes e \iff \Probab{\varphi} > \Probab{e} \text{ and } \big[\Pcond{e}{\varphi}
  > \Pcond{e}{\~\varphi} \big]\, ;
\end{equation}
with respect to the example above, patterns \footnote{Note that
  the conjunction $\wedge$ in our setting is interpreted
  differently from the classical notion (and the one adopted in e.g., \cite{gerstung2009quantifying}) since $a \wedge b \causes
c$ implies $a \causes c$ and $b \causes c$ in our framework. See also
\cite{Beerenwinkel:2014eb}.
Moreover, note that the scope of
  this study is intentionally kept limited from further generalization
  of formul\ae\; i.e., we will not consider statements of the form
  $\varphi_i \causes \varphi_j$, where the rightmost argument is a
  formula too.}
  could be $a \vee b \causes
c$ and $a \oplus b \causes c$
  

In our framework  the problem of reconstructing a probabilistic
graphical model of progression reduces to the following:
for each input event $e$, assess a \emph{set of selectivity patterns}
$\{\varphi_1 \causes e, \ldots, \varphi_k \causes e\}$,
filter the spurious ones,
and combine the rest in a \emph{direct acyclic graph}
\Newrev{(DAG)}\footnote{\Newrev{A DAG is formed by a set of nodes
       and  oriented  edges connecting one node to another, such that
       there are no directed loops among them. See SI Section 1 for a
       technical definition.}},
augmented with logical symbols.
Notice that while we broke down the progression extraction into a series of
sub-tasks, the problem remains complex:
patterns are unknown, potentially spurious, and exponential in formula
size; data is noisy; patterns must allow for
``imperfect regularities'', rather than being strict\footnote{This
  statement implies that there could be samples -- i.e., patients -- contradicting
  a pattern which still remains valid at a population level. 
 For this reason a
  pattern $x \wedge y \causes z$ is sometimes called a ``noisy
  and''.}.
To summarize, in our setting we can model complex progression trajectories
with branches (i.e., events involved in various patterns), independent
progressions (i.e., events without common ancestors) and convergence
(via CNF formulas).  The framework we introduce here is highly versatile, and 
to the best of
our knowledge, it
infers and checks more complex claims than any cancer progression
algorithms described thus far
(cfr.,\cite{desper_1999,gerstung2009quantifying,Loohuis:2014im}).





\section{Methods}
\label{sect:methods}

\begin{figure*}[t]
  \begin{center}
    \includegraphics[width=0.80\textwidth]{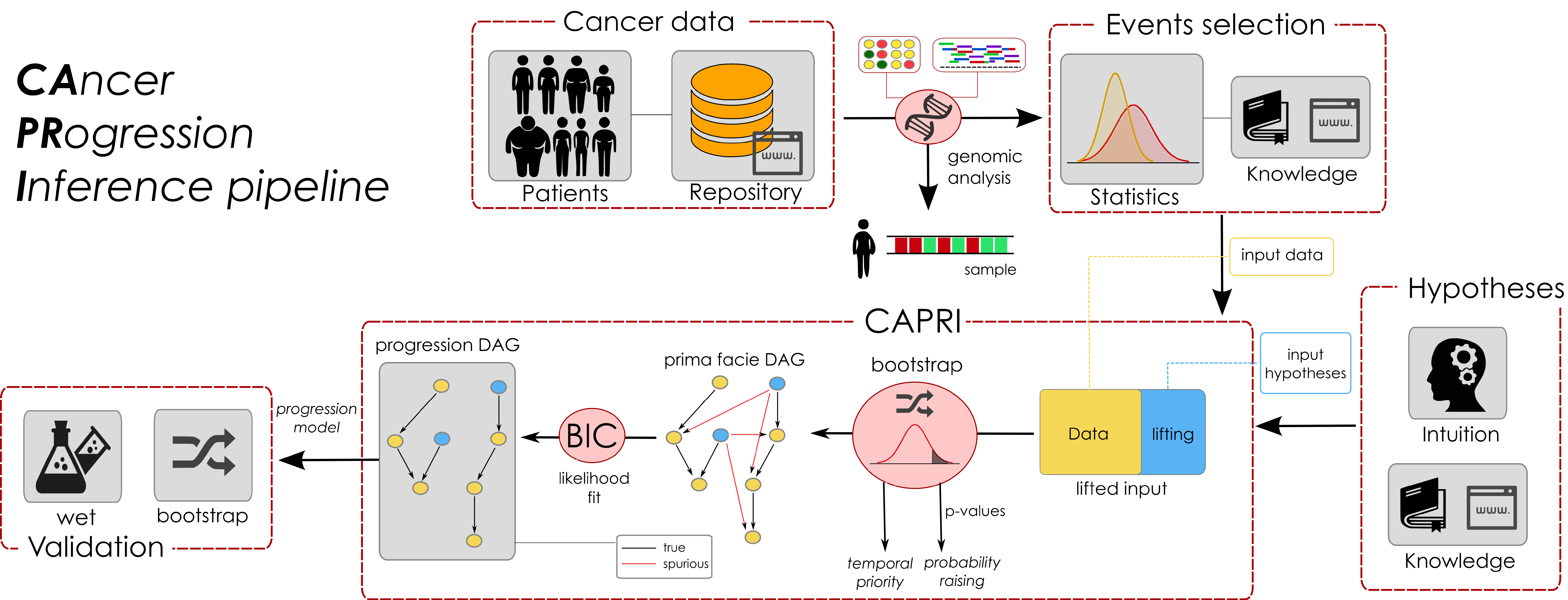}
  \end{center}
  \vspace*{.05in}
  \caption{\Newrev{\textbf{Data processing pipeline for cancer progression inference.} We sketch a pipeline to best exploit \ALGNAME{}'s ability to extract cancer progression models from cross-sectional data. Initially, one collects  \emph{experimental data} (which could be accessible through publicly available repositories such as TCGA) and performs  \emph{genomic analyses} to derive profiles of, e.g., somatic mutations or Copy-Number Variations for each patient. Then, statistical analysis and biological priors are used to select events relevant to the progression and imputable by  \ALGNAME{} - e.g., {\em driver mutations}. To exploit \ALGNAME{}'s supervised execution mode (see Methods) one can use further statistics and priors to generate \emph{patterns of selective advantage} - , e.g,  hypotheses of mutual exclusivity.  \ALGNAME{} can extract a progression model from these data   and assess various  \emph{confidence} measures on its constituting relations -  e.g., (non-)parametric  bootstrap and hypergeometric testing.  \emph {Experimental validation} concludes the pipeline.}}
  \label{fig:pipeline}
\end{figure*}

Building on the framework described in the previous section, we now describe the implementation of CAPRI's
building blocks. 
\Newrev{Notice that, in general, the inference of cancer progression models requires a complex \emph{data processing pipeline}, as summarized in Figure \ref{fig:pipeline};  its architecture optimally exploits \ALGNAME's efficiency.}

\paragraph{\textbf{Assumptions.}}
CAPRI relies on the following assumptions:
$i)$ Every pattern is expressible as a propositional CNF formula;
$ii)$ All events are persistent, i.e., an acquired mutation
cannot disappear;
$iii)$ All relevant events in tumor progression are observable, with
the observations describing the progressive
phenomenon in an essential manner (i.e., \emph{closed world} assumption,
in which all events `driving' the progression are detectable);
$iv)$ All the events have non-degenerate observed probability in
$(0, 1)$;
$v)$ All events are distinguishable, 
\Newrev{in the following sense: input alterations produce different profiles across input samples}. 
%
%
%
Assumptions $i$-$ii)$ relate to the framework derived in previous
section, while $iii)$ imposes an onerous burden on the
experimentalists, who must select the \emph{relevant}
genomic events to model\footnote{\Newrev{Theoretically, this assumption 
- common to other Bayesian learning problems -  is necessary to prove CAPRI's ability to extract 
the exact model in the optimal case of infinite samples. Practically,
as all \emph{relevant} events are hardly
selectable a priori and sample size is finite, further statistics can be used to select the most relevant driver alterations -- see
also Section~\ref{sect:discussion}, Results and Discussion. Nonetheless, CAPRI can provide significant results even if this assumption 
is not or cannot be verified.}}. \Newrev{Assumption $iv)$ relates instead to the statistical 
distinguishability of the input events (see the next section on CAPRI's Data Input).
}.

\paragraph{\textbf{Trading Complexity for Expressivity.}}
%
To automatically extract
the patterns that underly a progression model, one may try to adopt a brute-force
method of enumerating and testing all possibilities. This strategy is computationally intractable, however, since the number of (distinct)
(sub)formul\ae\ grows exponentially with the number of events included
in the model.  
Therefore, we need to
exploit certain properties of the $\causes$ relation whenever possible, and trade
expressivity for complexity in other cases, as explained below.

Note that \emph{singleton} and \emph{co-occurrence} ($\wedge$)
types of patterns are amenable to \emph{compositional reasoning}: if
$i_1 \wedge \ldots \wedge i_k \causes j$ then, for any $p = 1,\ldots, k$,
$i_p \causes j$.  This observation leads to the following
straightforward strategy of evaluating every
conjunctive (and henceforth singleton) relation using a pairwise-test
for the selectivity relation (see Figure~\ref{fig:single-cause}).

Unfortunately, it is easy to see that this reasoning fails to generalize for CNF patterns: e.g., when the
pattern contains disjunctive operators ($\vee$).  
As an example,
consider pattern  $a \vee b
\causes c$, in a cancer where
$\{a, \neg b\}$ progression to $c$ is more prevalent
than $\{\neg a, b\}$ and $\{ a, b\}$. In this case, considering
sub-formulas only we might find
$ a \causes c$ but miss
$b \causes c$ because the probability of
mutated $b$ is smaller than that of $c$,  
thus invalidating condition $(1)$ of relation $\causes$.  Notice that in extreme
situations, when the data is very noisy,  the algorithm may even ``invert'' the selectivity
 relation to $ c \causes b$.

This difficulty is not a peculiarity of our framework, but rather
intrinsic to the problem of extracting complex ``causal networks''
(cfr., \cite{pearl1988probabilistic,pearl,kleinberg-thesis}). To handle this situation, CAPRI adapts a strategy that trades
complexity for expressivity: 
\Newrev{the resulting inference procedure},
Algorithm~\ref{alg:inference}, can be executed in two modes: 
 unsupervised and supervised. In the former, inferred 
patterns of confluent progressions are constrained to co-occurrence 
types of relations, in the latter CAPRI can test  more complex patterns, 
i.e., disjunctive or \Newrev{``mutual exclusive''} ones, provided they are  given as prior
hypotheses. In both cases, CAPRI's complexity -- 
studied in next sections -- is quadratic both in the number of events
and hypotheses.

\paragraph{\textbf{Data Input (Step 1).}} 
%
%
\ALGNAME{} (cfr., Algorithm~\ref{alg:inference}) requires an input set $G$
of $n$ events, i.e., genomic alterations, and $m$ cross-sectional
samples, represented as a dataset in an $m \times n$ binary matrix $D$,
in which an entry $D_{i,j} = 1$ if the event $j$ was observed in
sample $i$, and $0$ otherwise. \Newrev{Assumption $iv)$ is satisfied 
when all columns in $D$ differ - i.e., the alteration profiles  yield 
different observations.}

Optionally, a set of $k$ input hypotheses  $\Phi= \{ \varphi_1 \causes
e_1, \ldots, \varphi_k \causes e_k\}$, where each $\varphi_i$ is a
well-formed\footnote{Formally, we require that $\varphi_i \not
  \sqsubseteq e_i$, where $\sqsubseteq$ represents the usual \emph{syntactical}
  ordering relation among \at{} events and formulas, and disallows for
  example 
$a \vee b \causes a$.} CNF
formula.
Note that we advise that the algorithm be used in the following regime~\footnote{
\Newrev{In the current biomedical setting, the number of samples ($m$) is usually in the hundreds, while number of possible mutations ($n$) and hypotheses ($k$), absent any pre-processing, could be large, thus violating the assumption; in these cases, we rely on various commonly used pre-preprocessing filters to limit $n$ to driver mutations, and $k$ to simple hypotheses involving the driver mutations. However, in the future as the number of samples increases, we envision a more agnostic application}.
}: $k+n \ll m$.

\paragraph{\textbf{Data Preprocessing (Lifting, step 2).}}
When input hypotheses are provided (e.g., by a domain expert), \ALGNAME{} first performs a
\emph{lifting operation}
  over
$D$ to permit direct inference of complex selectivity relations over a
joint representation, which involve input events as well as the hypotheses.  Lifting
operation evaluates  each input CNF formula -- for all  input
hypotheses in $\Phi$ -- and outputs a lifted matrix $D(\Phi)$ to be processed further as in step 1.
As an example, consider hypothesis
$a \oplus b \causes c$
lifted input matrix $D$ is:

\begin{scriptsize}
\[
D(\Phi) = \left[
  \begin{array}{ccc|c}
    a & b  & c & a \oplus b \causes c\\ \hline
    1 & 1 & 1 & 1 \oplus 1 = \mathbf{0}\\ 
    1 & 0 & 1 & 1 \oplus 0 = \mathbf{1}\\ 
    0 & 1 & 0 & 0 \oplus 1 = \mathbf{1}\\ 
    1 & 0 & 1 & 1 \oplus 0 = \mathbf{1}\\ 
  \end{array} 
\right]\, .
\] 
\end{scriptsize}

\noindent
Note that the first row (profile 
$\{a, b, c\}$
) contradicts the hypothesis,
while all other rows support it.

\paragraph{\textbf{Selectivity Topology (steps 3, 4, 5).}}
%
We exploit a compositional approach to 
test CNF hypotheses as follows: the disjunctive
relations are grouped, and treated as if they were individual objects
in $G$. For example, when a formula $\varphi \causes d$ where $\varphi= (a \vee b) \wedge c$
is considered, we assess $\varphi \causes d$ as whether $(a \vee b)
\causes d$ and $c \causes d$ hold -- with the proviso that we treat
$(a \vee b)$ as an individual event. Formally, with $\conj{\varphi}$
we denote the disjunctive clauses in a CNF formula. 

Nodes in the reconstruction are 
all input events together with all the disjunctive
clauses of each input formula $\varphi$.  

Edges in the reconstructed DAG are patterns that satisfy 
both conditions (1) and (2) of the selectivity relation
$\causes$. Formally, CAPRI includes an edge
between two nodes $\varphi$ and $j$
only if both $\Gamma_{\varphi,j} = \Probab{\varphi} - \Probab{j}$ and $\Lambda_{\varphi,j} =
      \Pcond{j}{\varphi} - \Pcond{j}{\~ {\varphi}}$ are strictly positive. Note
      that $\varphi$ can be both a disjunctive clause as well as a
      singleton event. 
%
A function $\pi(\cdot)$ assigns a parent to 
each node that is not an input formula. 
Note that this approach works efficiently by nature of the lifted
representation of $D$. 
The reconstructed DAG contains all
the true positive patterns, with respect to $\causes$, plus spurious
instances of $\causes$ which CAPRI subsequently removes in step~6
\Newrev{(cfr., the Supplementary Material for a proof of this statement)}.

Note that $\mathcal{D}$ can be readily interpreted as a
probabilistic graphical model, once it is augmented with a labeling
function $\alpha: N \to [0,1]$\Newrev{, where $N$ is the set of nodes
  -- i.e., the genetic alterations --} such that $\alpha(i)$ is the
\emph{independent probability} of observing mutation $i$ in a sample,
whenever \emph{all of its parent} mutations (i.e., $\pi(i)$) are
observed (if any). Thus $\mathcal{D}$  induces a \emph{distribution}
of observing a subset of events in a set of samples (i.e., a
probability of observing a certain  \emph{mutational profile} in a
patient).

\paragraph{\textbf{Maximum Likelihood Fit (step 6).}}
As the selectivity relation provides only a \emph{necessary}
condition, we must filter out all of its \emph{spurious instances}
that might have been included in $\mathcal{D}$ (i.e., the possible
\emph{false positives}).

For any selectivity structure, spurious claims contribute to
a reduction in the \emph{likelihood-fit} relative to true patterns. 
Thus, a
standard maximum-likelihood fit can be used to select and prune the
selectivity DAG (including a \emph{regularization term} to avoid
over-fitting\footnote{\Newrev{In principle other regularisation strategies common
to Bayesian learning could be used, e.g., Akaike information criterion (see 
\cite{carvalho2009scoring} and references therein). 
In this paper, we prefer to work with BIC which, in general, 
trades model complexity to reduce false positives rate.}}). 
Here, we  adopt the \emph{Bayesian
  Information Criterion} (BIC), which implements \emph{Occam's razor}
by combining log-likelihood fit with a \emph{penalty criterion}
proportional  to the $\log$ of the DAG size via \emph{Schwarz
  Information Criterion} (see \cite{bic_1978}). The BIC score is defined
as follows.
%
\begin{equation}\label{eq:bic}
  \textsc{bic}\left(\mathcal{D},\, D(\Phi)\right)
  = \mathcal{LL}\left(\mathcal{D},\, D(\Phi)\right) - \dfrac{\log m}{2}\, \mathrm{dim}(\mathcal{D}). 
\end{equation}
Here, $D(\Phi)$ is the lifted input matrix, $m$ denotes the number of
samples and  $\text{dim}(\mathcal{D})$ is the number of parameters in
the model $\mathcal{D}$. Because, in general, $\text{dim}(\cdot)$
depends on the number of parents each node has, it is a good metric
for model complexity. Moreover, since each edge added to $\mathcal{D}$
increases model complexity, the regularization term based on
$\text{dim}(\cdot)$ favors graphs with fewer edges and, more
specifically, fewer parents for each node.

At the end of this step, $\mathcal{D}$ and the labeling function are
modified accordingly, based on the result of BIC regularization. By
collecting all the incoming edges in a node it is possible to extract
the patterns, which have been selected by CAPRI as the positive ones.

\begin{algorithm}[t]
  \caption{\emph{\underline{CA}ncer \underline{PR}ogression
      \underline{I}nference} (\ALGNAME{})}
  \label{alg:inference}
  \begin{algorithmic}[1]
    \begin{scriptsize}

      \STATE \textbf{Input:} {A set of events $G=\{g_1,\ldots,g_n\}$,
        a matrix $D \in \{0,1\}^{m \times n}$} and $k$ CNF causal
      claims $\Phi=\{\varphi_1 \causes e_1, \ldots, \varphi_k \causes
      e_k\}$ where, for any $i$, $e_i \not \sqsubseteq \varphi_i$ and $e_i \in G$;

      \STATE [\emph{Lifting}] Define the \emph{lifting of $D$ to
        $D(\Phi)$} as the augmented matrix
      \begin{align*}\label{eq:lifting}
        D(\Phi) = 
        \left[ \begin{array}{ccc|ccc}
                 D_{1,1}  & \ldots & D_{1,n} & \varphi_{1}(D_{1,\cdot}) & \ldots &\varphi_{k}(D_{1,\cdot}) \\
                 \vdots & \ddots & \vdots & \vdots & \ddots & \vdots \\
                 D_{m,1}  & \ldots & D_{m,n} & \varphi_{1}(D_{m,\cdot}) & \ldots &\varphi_{k}(D_{m,\cdot}) \\
               \end{array}
        \right].
      \end{align*}
      by adding a column for each $\varphi_i \causes c_i \in \Phi$,
      with $\varphi_i$ evaluated row-by-row. Define then the coefficients
      $\Gamma_{i,j} = \Probab{i} - \Probab{j}$ and $\Lambda_{i,j} =
      \Pcond{j}{i} - \Pcond{j}{\~ {i}}$ pairwise over $D(\Phi)$;

      \STATE[\emph{DAG nodes}] Define the set of nodes $N = G \cup
      \left(\bigcup_{\varphi_i} \conj{\varphi_i}\right)$ which
      contains both input events and the disjunctive clauses in every
      input formula of $\Phi$.

      \STATE[\emph{DAG edges}] Define  a parent function $\pi$ where
      $\pi(j \not \in G) = \emptyset$ -- avoid edges incoming in a
      formula 
        \footnotemark -- and
      \begin{align} 
        \pi(j \in G) & = 
                       \left\{ i \in G \mid \Gamma_{i,j}, \Lambda_{i,j} >0 \right\} \nonumber\\
                     & \cup \left\{ \conj{\varphi}
                       \mid \Gamma_{\varphi,j}, \Lambda_{\varphi,j} > 0, \;
                       \varphi \causes j \in \Phi \right\}.
      \end{align}
      Set the DAG to $\mathcal{D} = (N,\pi)$.
 
      \STATE[\emph{DAG labeling}] Define the labeling $\alpha$ as follows
      \[
      \alpha(j) = \begin{cases}
        \Probab{j}, & \text{ if } \pi(j) =\emptyset  \text{ and } j \in G; \\
        \Pcond{j}{ i_1 \wedge \ldots \wedge i_n},  & \text{ if } \pi(j) = \{i_1, \ldots, i_n\}.
      \end{cases}
      \]


      \STATE[\emph{Likelihood fit}] Filter out all spurious causes
      from $\mathcal{D}$ by likelihood fit with the regularization BIC
      score and set $\alpha(j)=0$ for each removed edge.


      \STATE \textbf{Output:} the DAG $\mathcal{D}$ and $\alpha$;
    \end{scriptsize}
  \end{algorithmic}
\end{algorithm}
\footnotetext{Although \ALGNAME{} is equipped with bootstrap testing it is still possible to encounter various
degenerate situations. In particular, for some pair of events it could be that temporal priority cannot be satisfactorily resolved, i.e. there is no significant $p$-value for any edge orientation. Thus, loops might be  present in the inferred prima facie topology. Nonetheless, some of these could be still disentangled by probability raising, while some might remain, albeit rarely. To remove such edges we suggest to proceed as follows: $(i)$  sort these edges according to their $p$-value (considering both temporal priority and probability raising), $(ii)$ scan the sorted list in decreasing order of confidence, $(iii)$ remove an edge if it forms a loop.}

\paragraph{\textbf{Inference Confidence: Bootstrap and Statistical
    Testing.}}
%
%
To infer confidence intervals of the selectivity relations $\causes$,
CAPRI employs \emph{bootstrap with rejection resampling} as follows\Newrev{, by estimating} a distribution of the marginal and joint probabilities. For
each event, $(i)$ CAPRI samples with repetitions rows from the input
matrix $D$ (bootstrapped dataset), $(ii)$ CAPRI next estimates the
distributions from the observed probabilities, and finally, $(iii)$ CAPRI
rejects values which do not satisfy $0 < \Probab{i} < 1$ and $\Pcond{i}{j}
< 1 \vee \Pcond{j}{i} < 1$, and iterates restarting from $(i)$. We stop
when we have, for each distribution, at least $K$ values (in our case
$K = 100$). Any
inequality (i.e., checking temporal priority and probability
raising) is estimated using the non-parametric Mann-Whitney U test\footnote{The Mann-Whithney U test is a rank-based non-parametric statistical
  hypothesis test that can be used as an alternative to the Student's
  t-test and is particularly useful if data are not normally distributed.} with $p$-values
set
to $0.05$.
We compute
confidence $p$-values for both temporal priority and probability
raising using this test, which need not assume Gaussian distributions
for the populations.

Once a DAG $\mathcal{D}$ is inferred both \emph{parametric and
  non-parametric bootstrapping methods} can be used to assign a
confidence level to its respective pattern and to the overall  model.
Essentially, these tests consist of using the reconstructed model (in
the parametric case), or the probabilities observed in the dataset (in
the non-parametric case) to generate new synthetic datasets, which are
then reused to reconstruct the progressions (see,
e.g., \cite{efron_2013} for an overview of these methods). The
confidence is estimated by the number of times the DAG or any instance of
$\causes$ is reconstructed from the generated data.

\paragraph{\textbf{Complexity, Correctness and Expressivity.}}
CAPRI has the following asymptotic complexity (Theorem 1, SI Section 2):

%
$(i)$ Without input hypotheses  the execution is self-contained and
polynomial in the size of $D$. 

$(ii)$ In addition to the above cost, \ALGNAME{} tests input hypotheses of
$\Phi$  at a polynomial cost in the size of $|\Phi|$. In this case,
however,  its complexity may range over many orders of magnitude
depending on the structural complexity of the input set $\Phi$
\Newrev{consisting of hypotheses}. 

\Newrev{An empirical analysis of the execution time of CAPRI and the competing techniques on synthetic datasets is provided in the SI, Section 3.5}.

\ALGNAME{} is a \emph{sound and complete} algorithm, and its
expressivity in terms of the inferred patterns is proportional to the
hypothesis set $\Phi$ which, in turn, determines the complexity of the
algorithm. With a proper set of input hypothesis, \ALGNAME{} can infer
all (and only) the true patterns from the data, filtering out all the
spurious ones (Theorem 2, SI Section 2).
  Without
hypotheses, besides singleton and co-occurrence, no other patterns can
be inferred (see Figure \ref{fig:single-cause}). Also, some of these
claims might be spurious in general for more complex (and unverified) CNF
formula (Theorem 3, SI Section 2).






\section{Results and Discussion}
\label{sect:discussion}

\begin{figure*}
  \begin{center}
    \includegraphics[width=0.80\textwidth]{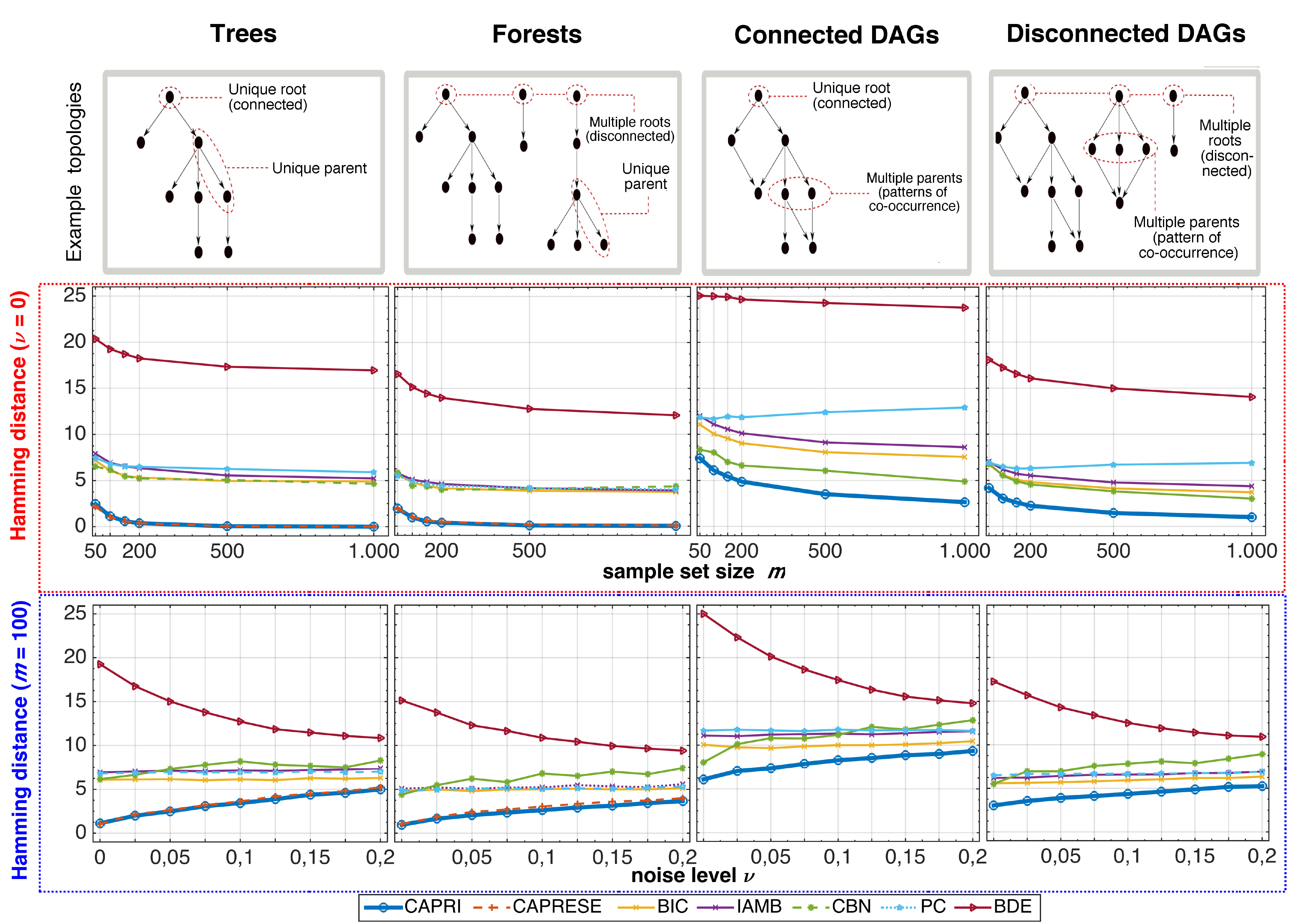}
  \end{center}
  \vspace*{.05in}
  \caption[Comparative Study]{\textbf{Comparative Study}. 
  Performance and accuracy of \ALGNAME{}
    (unsupervised execution) and other algorithms, IAMB, PC, BIC, BDE,
    CBN and CAPRESE{}, were compared using synthetic datasets sampled by a
    large number of randomly parametrized progression models -- \emph{trees},
    \emph{forests}, \emph{connected} and \emph{disconnected DAGs}, which
    capture different aspects of confluent, branched and heterogenous
    cancer progressions.  For each of those, $100$ models with $n=10$
    events were created and $10$ distinct datasets were sampled by each
    model. Datasets vary by number of samples $(m)$ and level of noise
    in the data $(\nu)$ -- see the Supplementary Information file for details.
    (\emph{Red box}) Average \emph{Hamming distance} (HD) -- with $1000$
    runs -- between the reconstructed and the generative model, as a
    function of dataset size $(m \in \{50, 100, 150, 200, 500, 1000
    \})$, when data contain no noise $(\nu = 0)$.  The lower the HD,
    the smaller is the total rate of mis-inferred selectivity relations
    among events. (\emph{Blue box}) The same is shown for a fixed sample
    set size $m = 100$ as a function of noise level in the data $(\nu
    \in \{0, 0.025, 0.05, \cdots, 0.2\})$ so as to account for input
    \emph{false positives} and \emph{negatives}. See SI Section 3 
for more extensive results on precision and recall scores and also including additional
    combinations of noise and samples as well as experimental settings.}
  \label{fig:performance}
\end{figure*}

To determine CAPRI's relative accuracy (true-positives and
false-negatives) and performance compared to the state-of-the-art
techniques for \emph{network inference}, we performed extensive
\emph{simulation experiments}. From a list of potential
competitors of CAPRI, we  selected:  \emph{Incremental Association
  Markov Blanket} (IAMB, \cite{tsamardinos2003algorithms}), the
\emph{PC algorithm} (see \cite{spirtes2000causation}), \emph{Bayesian
  Information Criterion} (BIC, \cite{bic_1978}), \emph{Bayesian
  Dirichlet with likelihood equivalence} (BDE, \cite{bde_1995})
\emph{Conjunctive Bayesian Networks} (CBN,
\cite{gerstung2009quantifying}) and \emph{Cancer Progression Inference  with Single Edges} (CAPRESE, \cite{Loohuis:2014im}). These
algorithms constitute a rich landscape of structural methods
(IAMB and PC), likelihood scores (BIC and BDE) and hybrid approaches
(CBN and CAPRESE).

%
\Newrev{Also, we applied CAPRI to the analysis of an atypical Chronic Myeloid Leukemia dataset
of somatic mutations — with data based on
\cite{piazza_nat_gen}.}

\subsection{Synthetic data}
\label{sect:comparison}

We performed extensive tests on a large number of \emph{synthetic
  datasets} generated by randomly parametrized progression models with
distinct key features, such as the presence/absence of: $(1)$
\emph{branches}, $(2)$ \emph{confluences with patterns of
  co-occurrence}, $(3)$ \emph{independent progressions} (i.e.,
composed of disjoint sub-models involving distinct sets of events).
Accordingly, we distinguish four classes of generative models with
increasing complexity and the following features:
\begin{center} \footnotesize
\begin{tabular}{c|c|c|c|c}
 & \emph{trees} & \emph{forests} & \emph{connected DAGs}  & \emph{disconnected DAGs} \\
 \hline
(1) & \okmark & \okmark & \okmark & \okmark \\
 \hline
(2) & \nomark & \nomark & \okmark & \okmark \\
 \hline
(3) & \nomark & \okmark & \nomark & \okmark \\
 \end{tabular}
\end{center}

%
The choice of these different type of topologies is not a mere
technical exercise, but rather it is motivated, in our application of
primary interest, by \emph{ heterogeneity of cancer cell types} and
\emph{possibility of multiple cells of origin}.

To account for \emph{biological noise} and \emph{experimental errors}
in the data we introduce a parameter $\nu \in (0,1)$ which represents
the probability of  each entry to be random in $D$, thus representing
a \emph{false positive}  ($\epsilon_+$)  and a \emph{false negative}
rate ($\epsilon_-$): $\epsilon_+ =\epsilon_- = {\nu}/{2}\,$.  The
noise level complicates the inference problem, since samples generated
from such topologies will likely contain sets of mutations that are 
correlated but causally irrelevant.

To have  reliable statistics in all the tests, $100$ distinct
progression models per topology are generated and, for each model, for
every chosen combination of sample set size $m$ and noise rate $\nu$,
$10$ different datasets are sampled (see SI Section 3 
for our synthetic data generation methods).


Algorithmic performance was evaluated using the metrics \emph{Hamming
  distance} (HD), \emph{precision} and
\emph{recall}, as a function of dataset size, $\epsilon_+$ and
$\epsilon_-$.  HD measures the \emph{structural similarity} among the
reconstructed progression and the generative model in terms of the
minimum-cost sequence of node edit operations (inclusion and exclusion)
that transforms the reconstructed topology into the generative
one\footnote{This measure corresponds to the sum of false positives
  and false negative and, for a set of $n$ events, is bounded above by
  $n(n-1)$ when the reconstructed topology contains all the false
  negatives and positives.}. Precision and recall are defined as
follows: \emph{precision} = TP/({TP} + {FP}) and \emph{recall} =
TP/({TP} + {FN}), where {TP} are the \emph{true positives} (number of
correctly inferred true patterns), {FP} are the \emph{false positives}
(number of spurious patterns inferred) and {FN} are the \emph{false
  negatives} (number of true patterns that are \emph{not} inferred
). The closer both precision and recall are to $1$, the
better.

%
In Figure \ref{fig:performance} we show the performance of CAPRI and
of the competing techniques, in terms of Hamming distance, on datasets
generated from models with $10$ events and all the four different
topologies.  In particular, we show the performance: $(i)$ in the case
of noise-free datasets, i.e., $\nu = 0$ and different values of the
sample set size $m$ and $(ii)$ in the case of a fixed sample set size,
$m = 100$ (size that is likely to be found in currently available
cancer databases, such as TCGA (cfr., \cite{tcga})) and different
values of the noise rate $\nu$. As is evident from Figure \ref{fig:performance} \ALGNAME{}
outperforms all the competing techniques with respect to all the
topologies and all the possible combinations of noise rate and sample
set size, in terms of average Hamming distance (with the only
exception of CAPRESE in the case of tree and forests, which displays a
behavior closer to CAPRI's). The analyses on precision and
recall 
 display consistent results (SI Section 3). In other words,
we demonstrate on the basis of extensive synthetic tests that CAPRI requires
a much lower number of samples than the other techniques in order to
converge to the real generative model and also that it is much more
robust even in the presence of significant amount of noise in the data, irrespective of the
underlying topology.

See SI Section 3
for a more complete description 
of the performance evaluation for all the
analyzed combinations of parameters. There, we have shown that \ALGNAME{} is highly
effective when the co-occurrence constraint on confluences is
relaxed to  \emph{disjunctive} patterns, 
\emph{ even if no input hypotheses are provided}, i.e.,
$\Phi=\emptyset$. This result hints at CAPRI's robustness to infer patterns
with imperfect regularities. Finally, we also show that \ALGNAME{} is
effective in inferring synthetic lethality relations in this case using the
operator $\oplus$ as introduced in
\Newrev{Section~\ref{sect:approach}, Approach}; when a
combination of mutations in two or more genes leads to cell death,
while separately, the mutations are viable. In this case, candidate
relations are directly input as $\Phi$. 

\begin{figure*}[t]
  \begin{center}
    \includegraphics[width=0.46\textwidth]{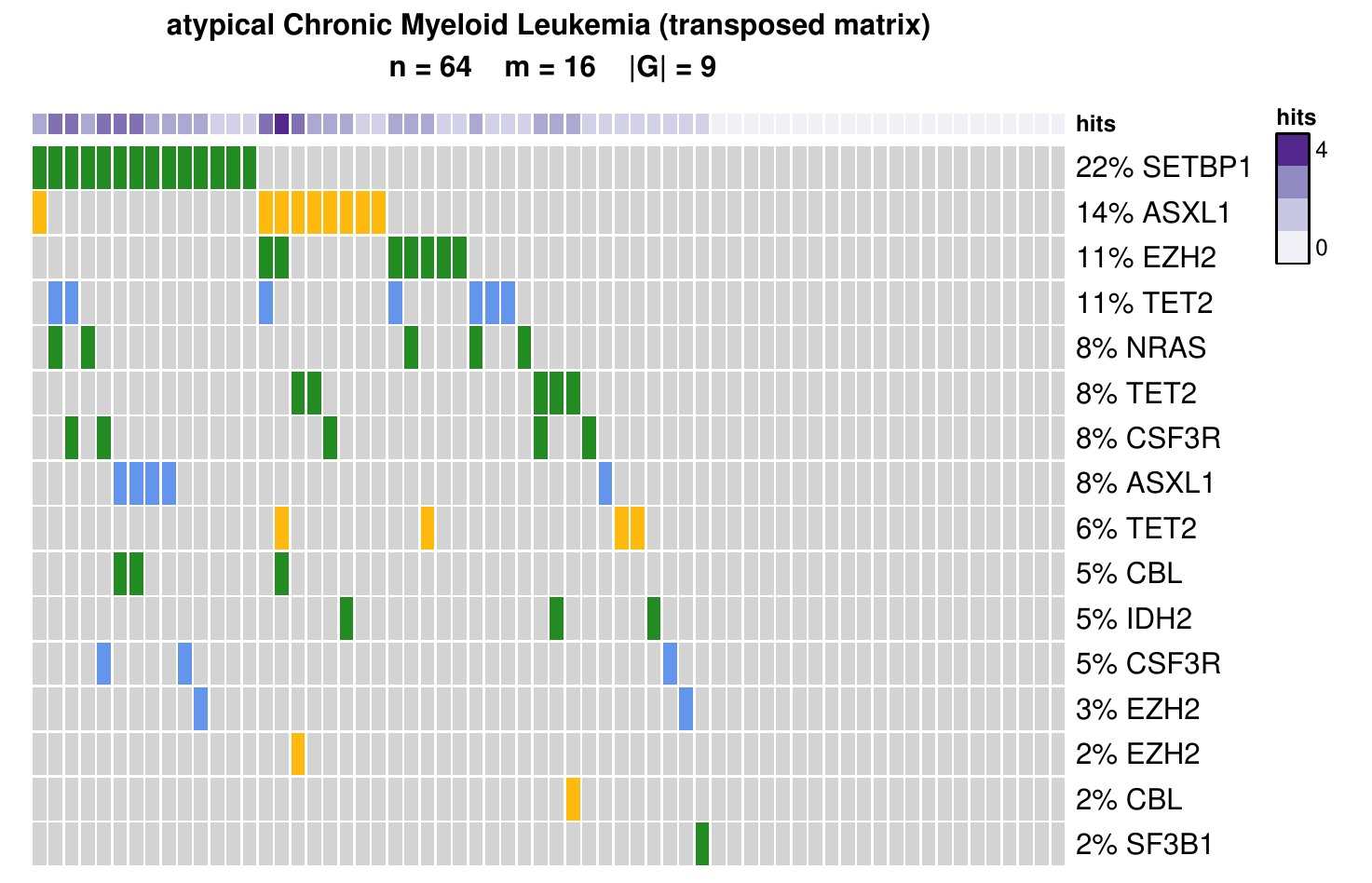}
    \includegraphics[width=0.53\textwidth]{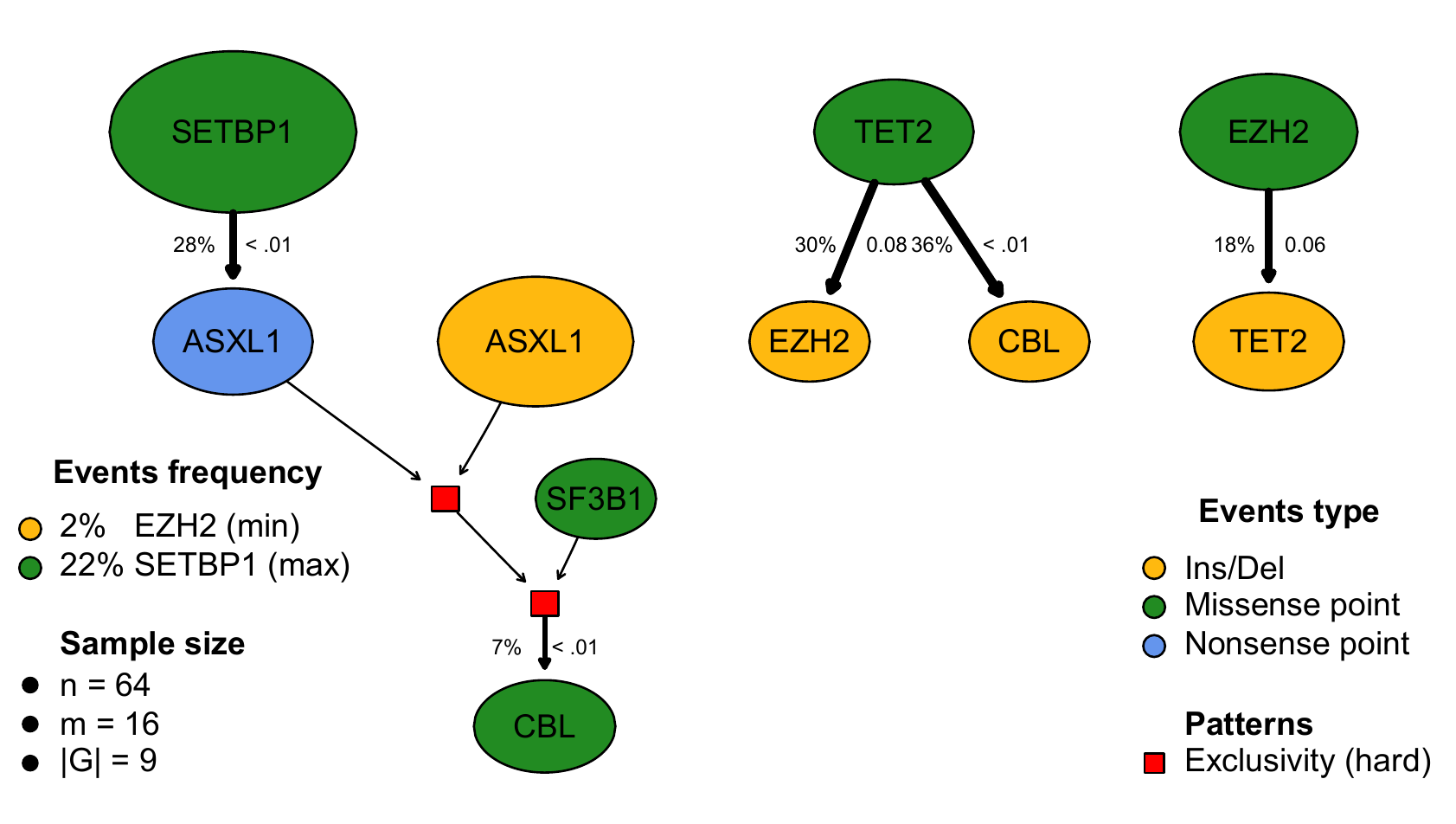}
  \end{center}
  \vspace*{.05in}
  \caption[Atypical Chronic Myeloid Leukemia ]{\Newrev{\textbf{Atypical Chronic Myeloid Leukemia}. 
  {\em(left)} Mutational profiles of $n=64$ \aCML{} patients - exome sequencing in \cite{piazza_nat_gen} - with alterations in $|G|=9$  genes with either mutation frequency $>5\%$ or 
  belonging to an hypothesis inputed to CAPRI (SI Section 4). Mutation types are classified as \emph{nonsense point}, \emph{missense point} and  \emph{insertion/deletions}, yielding $m=16$ input events. Purple annotations report the frequency of mutations per sample. {\em(right)} Progression model inferred by \ALGNAME{} in supervised mode.  Node size is proportional to  the marginal probability of each event,  edge thickness to the confidence estimated with $1000$ non-parametric bootstrap iterations (numbers shown leftmost of every edge). The p-value of the hypergeometric test is displayed too. Hard exclusivity patterns inputed to CAPRI are indicated as red squares. Events without inward/outward edges are not shown.}}
\label{fig:aCML_model}
\end{figure*}

\subsection{Atypical Chronic Myeloid Leukemia (aCML)}
\label{sect:aCML}

\Newrev{
As a case study, we applied CAPRI to the mutational profiles 
of $64$ \aCML{} patients described in \cite{piazza_nat_gen}. Through exome
sequencing, the authors identify a recurring \emph{missense
  point  mutation} in the \emph{SET-binding protein 1} (\setbp{}) gene as
a novel \aCML{} marker.

Among all the genes present in the dataset by Piazza {\em et al.},
we selected those either $(i)$ mutated - considered any mutation type
- in at least $5\%$ of the input samples ($3$ patients), or $(ii)$
hypothesised to be part of a functional \aCML{} progression pattern in
the literature \footnote{Two  \emph{hard exclusivity} patterns - i.e.,
  mutual exclusivity with ``xor'' - were tested, involving the
  mutations of: $(i)$ genes \asxl{} and  \sfb{} (see
  \cite{lin2014sf3b1}), which is present in the inferred progression
  model in Figure  \ref{fig:aCML_model}, and $(ii)$ genes \tet{} and
  \idh{} (see \cite{figueroa2010leukemic}). The syntax in which  the
  patterns are expressed is in the SI, Section 4.}. The input dataset
with selected events  is shown in Figure \ref{fig:aCML_model}; notice
that somatic mutations are categorised as \emph{indel},
\emph{missense point} and \emph{nonsense point} as in \cite{piazza_nat_gen}. In Figure \ref{fig:aCML_model} we show the
model reconstructed by CAPRI (supervised mode, execution time
$\approx 5$ seconds) on this dataset, with confidence 
assessed via $1000$ non-parametric bootstrap iterations. The  model highlights several non trivial selectivity relations involving genomic events relevant to \aCML{} development.

First, CAPRI predicts a progression involving  mutations in \setbp{},
\asxl{}  and \cbl{}, consistently with the recent study by \cite{meggendorfer2013setbp1}, in which these genes were shown to be highly correlated  and possibly functioning in a synergistic manner for \aCML{} progression. 
Specifically, CAPRI predicts  a selective advantage relation between
missense point mutations in \setbp{} and nonsense point mutations in
\asxl{}. This is in line with recent evidence from
\cite{inoue2014setbp1} suggesting that \setbp{} mutations are enriched
among \asxl{}-mutated \emph{myelodysplastic syndrome} (MDS) patients,
and \emph{in-vivo} experiments point to a driver role of \setbp{} for
that leukemic progression. 
Interestingly,  our model seems also to suggest a different role of \asxl{}
\emph{missense} and \emph{nonsense} mutation types in the
progression, yet more extensive studies (e.g., prospective or systems biology explanation) are needed to corroborate this hypothesis. 


Among the hypotheses given as input to CAPRI, the algorithm seems to suggest that the exclusivity pattern among \asxl{} and \sfb{} mutations selects for \cbl{} missense point mutations. The role of the \asxl{}/\sfb{} exclusivity pattern is consistent with the study of \cite{lin2014sf3b1} which shows that, on a cohort of 479 MDS patients, mutations in \sfb are inversely related to \asxl{} mutations. 

Also, in \cite{abdel2012asxl1} it was recently shown that \asxl{} mutations, in patients with MDS, \emph{myeloproliferative neoplasms} (MPN) and \emph{acute myeloid leukemia}, most commonly occur as  nonsense and insertion/deletion  in a clustered region adjacent to the highly conserved \emph{PHD} domain (see \cite{gelsi2009mutations}) and that mutations of any type eventually result in a loss of  \asxl{} expression. This observation is consistent with the exclusivity pattern among \asxl{} mutations in the reconstructed model, possibly suggesting alternative trajectories of somatic evolution for  \aCML{} (involving either \asxl{} nonsense or indel mutations).

Finally, CAPRI predicts  selective advantage relations among \tet{}
and \ezh{} missense point and indel mutations. Even though the limited sample size does not allow to draw definitive conclusions on the
ordering of such alterations, we can hypothesize that they may play a synergistic role in \aCML{} progression. Indeed, \cite{muto2013concurrent}  suggests that the concurrent loss of \ezh{} and \tet{} might cooperate in the pathogenesis of myelodysplastic disorders, by accelerating the overall tumor development, with respect to both MDSs and \emph{overlap disorders} (MDS/MPN). 

}





\section{Conclusions}
\label{sect:conclusions}

The \emph{reconstruction of cancer progression models} is a pressing problem,
as it promises to highlight important clues about the evolutionary dynamics of
tumors and to help in better targeting therapy to the
tumor (see e.g.,~\cite{olde2014cancer}).
 In the absence of large longitudinal datasets, progression extraction
algorithms rely primarily on \emph{cross-sectional} input data, 
thus complicating the statistical inference problem. 

In this paper we presented CAPRI, a new algorithm (and part of the TRONCO
package) that attacks the progression model reconstruction problem by
inferring \emph{selectivity relationships} among
``genetic events'' and organizing them in a graphical model.  The
reconstruction algorithm draws its power from a combination of a
scoring function (using Suppes' conditions) and subsequent filtering and refining procedures, maximum-likelihood estimates
and bootstrap iterations.  We have shown that CAPRI outperforms a wide
variety of state-of-the-art algorithms. 
We note that CAPRI performs especially well in
the presence of noise in the data, and with limited sample size.
Moreover we note that, unlike other approaches, CAPRI can
reconstruct different types of confluent trajectories unaffected by the
irregularities in the data -- the only limitation being our ability to
hypothesize these patterns in advance.  We also note that CAPRI's overall
algorithmic complexity and convergence properties 
do offer several tradeoffs to
the user.

\Newrev{
Successful cancer progression extraction
is complicated by tumor heterogeneity: many tumor types have molecular
subtypes following different progression patterns.
For this reason, it can be advantageous to cluster patient samples
by their genetic subtype prior to applying CAPRI. Several tools have
been developed that address this clustering problem (e.g.,
Network-based stratification \cite{hofree2013network} or COMET from \cite{comet}). 
A related problem is the classification of
mutations into functional categories. 
In this paper, we have used genes with deleterious mutations
as driving events. However, depending on other criteria, such as the level of
homogeneity of the sample, the states of the progression can represent
any set of discrete states at varying levels of abstraction. 
Examples include high-level hallmarks of cancer
proposed by \cite{hanahan_hallmarks_2000, Hanahan_Weinberg_2011},  a set of
affected pathways, a selection of driving genes,  or a set of specific
genomic aberrations such as genetic mutations at a more mechanistic
level. 
}

We are currently using CAPRI to conduct a number of studies on
publicly available datasets (mostly from TCGA, \cite{tcga}) in
collaboration with colleagues from various institutions.
\Newrev{In this work we have shown the results of the reconstruction on the \aCML{} dataset published by \cite{piazza_nat_gen}, and in SI Section 4 we include a further example application on ovarian cancer (\cite{knutsen2005interactive}), as well as a comparative study against the competing techniques. 
Furthermore, we are currently extending our pipeline in order to include 
pre-processing functionalities, such as patient
clustering and categorization of mutations/genes into pathways 
(using databases such as the KEGG database (see \cite{kanehisa2000kegg}) and
functionalities from tools like Network-based clustering, due to \cite{hofree2013network}. 

}

Encouraged by CAPRI's ability to infer interesting relationships in a
complex disease such as aCML, we expect
\Newrev{that in the future} CAPRI will
help uncover relationships to aid our understanding of cancer and
eventually improve targeted therapy design.

\section*{Acknowledgements}

This research was funded by the NSF
grants CCF-0836649 and CCF-0926166 and by Regione Lombardia (Italy)
under the research projects RetroNet through the ASTIL Program
[12-4-5148000-40]; U.A 053 and Network Enabled Drug Design project
[ID14546A Rif SAL-7] Fondo Accordi Istituzionali 2009.

We also thank Francesca Ciccarelli, King's College
London, UK, and others for suggesting the ``selectivity advantage''
terminology. 
We would also like to thank all the participants of the Workshop and
School on \emph{Cancer, Systems and Complexity} held on Lake Como,
Italy for many fruitful discussions there 
(\texttt{csac.lakecomoschool.org}).
Finally, we are also indebted to Rocco Piazza, Universit\`{a} degli Studi di
Milano Bicocca, Italy, for all the data, insights and
patience in explaining to us the biology of aCML.

\bibliographystyle{ieeetr}
\bibliography{references}

\appendix

\newpage

\section{Models inferred with CAPRI} \label{sec:algorithms}

We define a {\em progression DAG} as a directed acyclic graph  ${\cal D}=(N, \pi)$, where $N\subseteq {\cal U}$ is the set of nodes (e.g., selected from a universe ${\cal U}$ of {\em mutations}  or propositional formulas) and $\pi:N\to \wp(N)$ is a function, which associates with each node $j$ its {\em parents} $\pi(j) \subseteq N$. 
We wish to study the cases where such a DAG can be seen as a “model” for the following classes of \emph{selectivity patterns}, expressed in conjunctive normal form (CNF). The symbol $\causes$ stands for the \emph{selectivity relation}. 

\begin{definition}[DAG patterns] \label{def:DAG-claims} A  ${\cal D}=(N, \pi)$ is a model for models the {\em patterns}  
\[
\bigcup_{j\in N} \Big\{(\cc_1 \wedge \ldots \wedge \cc_n) \causes j  \mid  \pi(j)=\{\cc_1, \ldots, \cc_n\} \Big\}\, ,
\] where  $\cc_1 \wedge \ldots \wedge \cc_n$ is a CNF formula (each clause $\cc_j$ is one of two kinds: either an atomic event or a disjunction of events).
 \end{definition} 

Each DAG induces a {\em distribution} of observing a subset of events in a set of samples (i.e., a probability of observing a certain  {\em mutational profile} in the context of our application).
\begin{definition}[DAG-induced distribution] \label{def:treedistrib} Let ${\cal D}=(N, \pi)$ be a DAG and $\alpha:N\to[0,1]$ a labeling function, $\cal D$ generates a distribution where the probability  of observing  $N^\ast \subseteq N$ events is
\begin{equation}
\Probab{N^\ast} = \prod_{x \in N^\ast} \alpha(x) \cdot 
\prod_{y \in N \setminus N^\ast}
 \Big[1-\alpha(y)\Big]
\end{equation}
whenever $x \in N^\ast$, $\pi(x) \subset N^\ast$, and $0$ otherwise. 
\end{definition}

Notice that this definition, as expected, is equivalent to the one used in~\cite{beerenwinkel_2007} and retains a  tree-induced distribution such as those used in~\cite{Loohuis:2014im,desper_1999,szabo}. Further, notice that a sample which contains an event but not all  of its parents has a zero probability, thus subsuming the conjunctive interpretation of DAGs, as the result of compositional reasoning to infer co-occurrence patterns. These kinds of samples, which represent ``irregularities'' with respect to $\cal D$, might be generated when adding false positives/negatives to the sampling strategy.

\section{Theorems}

The statements and proofs of the theorems mentioned in the main text follow.

\subsection{Complexity}

Let $\mathcal{U}$ denote the \emph{universe} of all possible patterns
over a set $G$ of $n$ events, as before. Since $|\mathcal{U}|$ is exponential in $|G|$, then the following theorem
holds.
\begin{myth}[Asymptotic complexity]
  \label{th:complexity}
  Let $|G|=n$ and $D \in \{0,1\}^{m\times n}$
  where $m\gg n$, and let $N$ be the  nodes in the DAG returned by
  \ALGNAME{},  the \emph{worst case} time and space complexity (ignoring the cost of bootstrap) of
  building a selectivity topology is:
  \begin{itemize}
  \item $\Theta(mn)$ time and $\Theta(n^2)$ space, if
      $\Phi=\emptyset$;
    \item ${\Theta}(|\Phi|mn)$ time and ${\Theta}(|\Phi|m)$ space, if
      $\Phi \subset \mathcal{U}$ and $|N| \ll m$ (i.e., there are
      sufficiently many samples to characterize the input hypotheses);

    \item $\mathcal{O}(2^{2^n})$ time and space, if $\Phi= \mathcal{U}$.
    \end{itemize}
Thus, the overall complexity of \ALGNAME{} is one of the above, as suitable in each case, plus
the complexity of likelihood fit with regularization.
\end{myth}
\begin{proof} Recall that $k=|\Phi|$, $n=|G|$ and $D \in \{0,1\}^{m\times n}$, thus $D(\Phi)$ has $K=(n+k) m$ entries. We now analyze the complexity of \ALGNAME{} step-by-step.

\begin{itemize}
\item The cost of lifting depends on the input set $\Phi$, if $\Phi=\emptyset$ it is ${\cal O}(1)$ both in time and space since $D(\emptyset)=D$. 

For non-empty sets, it is necessary to evaluate $k \cdot m$ entries, after each hypothesis $\varphi \causes e$ is evaluated. Given that every $\varphi$ has at worst $n$ events included, its evaluation cost is at most ${\cal O}(n)$, even if lazy evaluation is performed.
Thus, the cost of lifting is $\Theta(k \cdot m \cdot n)$, for a single bootstrap, which amplifies the bootstrap cost, as discussed in the previous section, and does so in a multiplicative fashion. In terms of space, if 
$\Phi\neq\emptyset$ the overhead is ${\Theta}(K)$  if one copies $D$ in $D(\Phi)$,  $\Theta(km)$ otherwise.

\item The cost of computing the parent function for the DAG requires a pair-wise calculation of the probabilistic scores, plus the cost of testing the $\sqsubseteq$ relation\footnote{Relation $\sqsubseteq$ represents the usual syntactical ordering relation among atomic events, e.g., $a$, $b$, and formulas, e.g., $a \sqsubseteq (a \vee b) \vee c \vee d$.}. Let $w=|N|$, where $N$ is the set of nodes in the DAG returned by \ALGNAME{}. 
The score matrices for temporal priority and probability raising are  $n\times w$, i.e., have columns for both atomic events and the disjunctive patterns in the formulas of $\Phi$, since \ALGNAME{} disregards patterns of the form $\varphi_i \causes \varphi_j$ and $a \causes \varphi$ (differently, it would have been $w \times w$). With the simplest membership test algorithm, checking whether an \at{} event is present in a patterns  is logarithmic in the size of the pattern, if we lexicographically order its atomic events, thus bounded from above by $\log n$.  Thus if  we perform lazy evaluation for $\sqsubseteq$  the total number of comparison to select the parent function is at most
\[
n[(n-1) + (w-n) \log n],
\] 
yielding a  $\Theta(n^2)$ cost in time and space, if $w-n$ is small (it is $0$ if $\Phi=\emptyset$), ${\cal O}(n(w-n)\log n)$ otherwise.  In terms of space,  the complexity is $\Theta(n[(n-1) + (w-n)])$, for a general $\Phi$.


\item As explained in \ALGNAME{}'s definition, sometimes, albeit extremely rarely, a few extra operations might have to be performed when degenerate scores and loops are present. The procedure we suggested in \ALGNAME{}'s definition requires sorting plus scan, thus its worst-case  time complexity is ${\cal O}(n\log n)$. Clearly, as this term is omitted in the worst-case complexity analysis of the steps discussed above, this unlikely scenario does not alter the complexity of the algorithm.

\item Note that the cost of this analysis does not include the cost of BIC/likelihood - or any regularization strategy one might adopt, as spelled out in the theorem statement\footnote{Since in the current version of CAPRI, the likelihood fit is computed by a \emph{hill climbing} heuristic algorithm, the overall cost of CAPRI is still polynomial.}.
\end{itemize}

The overall complexity follows, since:
\begin{itemize}
\item $\Phi=\emptyset$ then the major cost is that of evaluating $\Probab{\cdot}$ since usually $m\gg n$, thus $mn>n^2$. With regard to space, the only cost is that of book-keeping the scores.

\item Let $m \gg n$ and $w-n>k$, in this case since $km\gg n$ and, under the mild assumption that $m>w$ and that $k$ and $\log n$ are not relevant (in size) for $m$ and $w$, then  $km \gg (w-n)\log n$ which is  the  cost  of lifting; thus  is ${\Theta}(kmn)$ in time. Similarly, it follows that $mk\gg n[(n-1)+(w-n)]$.

\item By  computations similar to those carried out, it is indeed possible to see that  $\cal U$, which is clearly finite since $G$ is, grows {\em double-exponentially} in size with $|G|$  (i.e. the number of $n$-ary boolean functions, defined over the \at{} events  in any pattern, possibly with negated literals). Thus the bound follows. 
\end{itemize}
\end{proof}

\subsection{Correctness and expressivity}

Let $\mathcal{W} \subseteq \mathcal{U}$ be the set of true patterns,
which we seek to infer. Here, we investigate the relation between
$\mathcal{W}$ and the patterns retrieved by CAPRI, as  a function of
sample size $m$ and error present as false positives/negatives, which
are assumed to occur at rates $\epsilon_+$ and $\epsilon_-$.

Hereafter, $\Sigma$ denotes the set of patterns, implicit in the DAG
returned by our algorithm for an input set $\Phi$ and a matrix $D$; we
write this fact as $D(\Phi)\models \Sigma$.  We prove the following 
theorems\footnote{These results assume a BIC regularisation but
hold for any convergent regularization score.}.

\begin{myth}[Soundness and completeness]
  \label{th:soundness}
  Let the sample size $m \to \infty$ and the data be uniformly randomly
  corrupted by false positives and negatives rates
  $\epsilon_-=\epsilon_+\in[0,1)$. If the given input is a superset
  of the true patterns, then \ALGNAME{} reconstructs exactly the  true
  patterns in $\mathcal{W}$, that is, $\mathcal{W} \subset \Phi$ $\Rightarrow$
  $D(\Phi)\models \mathcal{W} \cap \Phi$.
\end{myth}
\begin{proof} 
We first prove the case with $\epsilon_+=\epsilon_-=0$, that is, the case where data have no noise. Some notations, used below: $(i)$ we denote with $\varphi \causes e$ true patterns (i.e. in $\cal W$), and $(ii)$ with $\varphi^\ast \causes e$ false ones. We divide the proof into several steps:
\begin{itemize}
\item First, we show that a selectivity DAG contains all the true patterns, which is
\[
\forall_{\varphi \causes e \in {\cal W}} \;  \pi(e) = \{\varphi\}\, .
\]
By the event-persistence property usually valid for cancer genomes (fixating mutations are present in the progeny of a clone) the occurring times satisfy $t_\varphi < t_e$ which, in a frequentist sense, implies $\Probab{\varphi} > \Probab{e}$. In addition, 
it holds by construction  that $\Pconj{\varphi}{e} = \Probab{e}$ when $\epsilon_+=\epsilon_-=0$, thus $\Pcond{e}{\varphi}= \Probab{e}/\Probab{\varphi}$, which is strictly positive since $\Probab{\varphi}$ and $\Probab{e}$ are, and that $\Pconj{\~\varphi}{e}=0$ , thus $\Pcond{e}{\~\varphi}=0$. Notice that $e \not \sqsubseteq \varphi$ by hypothesis.

\item Now, we show that it might contain also spurious patterns, which is
\[
\exists_{\varphi^\ast \causes e \not\in {\cal W}} \; \pi(e) \subseteq \conj{\varphi^\ast} \cup \{ \varphi^\ast\}\, .
\]
These  $\varphi^\ast \causes e$ are of two types: sub-formulas spurious or topologically spurious (which include transitivities, as we may recall). For the former case note that
\[
\forall_{\varphi \causes e \in {\cal W}} \; \forall_{\hat{\varphi}^\ast \in \conj{\varphi}}\; \hat{\varphi}^\ast \causes e \not \in \cal W,
\]
but satisfies both temporal priority and probability raising. Also, consider any other ${\hat{\varphi}^\ast}_{\star} \sqsubseteq \hat{\varphi}^\ast$  and note that even this might satisfy both temporal priority and probability raising.   For the latter case, it might be that there exists some other $\varphi^\ast$ such that, it is positively statistically correlated to a real pattern, and that might satisfy Suppe's conditions as well.
\end{itemize}
Thus, for any $e \in G$ such that $\varphi \causes e \in \cal W$
\[
\pi(e) =  \{ \varphi\}   \cup {\cal S},
\]
where $\cal S$ is a set of spurious patterns. We now examine the relation holding between the selectivity DAG and its modification performed via BIC. The derivations shown in the following hold regardless of the type of regularization which enjoys convergency.

We denote these DAGs as ${\cal D}_{\text{pf}}$ and ${\cal D}_{\text{BIC}}$.
\begin{itemize}
\item[$(i)$] First, we show that all true patterns in ${\cal D}_{\text{pf}}$ are in ${\cal D}_{\text{BIC}}$, i.e.
\[
\forall_{\varphi \causes e \in {\cal W}} \;  \pi_{\text{BIC}}(e) = \{ \varphi\} \, .
\]
Note that, although in general $\Pconj{a}{b} \leq \min \{ \Probab{a}, \Probab{b}\}$, for the true patterns the following holds: $\Pconj{\varphi}{e} = \Probab{e}$, when $\epsilon_+=\epsilon_-=0$; it is the maximum value for this joint probability, thus ensuring the maximum-likelihood fit. Thus the pattern   is maintained in ${\cal D}_{\text{BIC}}$.

\item[$(ii)$] Second, we need to show that if $\forall \varphi^\ast \causes  e \not \in {\cal W}$ but present in ${\cal D}_{\text{pf}}$, there exists a pattern $\varphi \causes  e \in {\cal W}$, which is present in ${\cal D}_{\text{pf}}$ and in ${\cal D}_{\text{BIC}}$ and any 
$\varphi^\ast \causes e$ is not in ${\cal D}_{\text{BIC}}$. 

Note that $\Pconj{\varphi}{e} = \Probab{e}$, as above. Instead,  $\Pconj{\varphi^\ast}{e} < \Probab{e}$ since it is spurious, hence 
$\Pconj{\varphi \wedge \varphi^\ast}{e} < \Pconj{\varphi}{e}$, thus the likelihood fit of $\varphi \causes  e$ is maximal with respect to any of the patterns $\varphi^\ast \causes  e$.
\end{itemize}
To extend the proof to $\epsilon_+=\epsilon_-\in[0,1)$ with uniform noise, it suffices to note that the marginal and joint probabilities change monotonically as a consequence of the assumption that the noise is uniform. Thus, all inequalities used in the preceding proof still hold, which concludes the proof.
\end{proof}

Notice that if it could be assumed that  $\Phi$ characterizes
$\mathcal{W}$ well, then all true patterns would be in $\Phi$, and the
corollaries below follows immediately.
\begin{myco}[Exhaustivity]
  Assuming the same hypothesis as the theorem agove,  $D(\mathcal{U})\models
  \mathcal{W}$. \flushleft{$\Box$}
\end{myco}

\begin{myco}[Least Fixed Point]
  $\mathcal{W}$ is the \emph{lfp} of the monotonic transformation
  \[
  \bigsqcup_{\Phi} D(\Phi)
  \equiv D\Big( \bigsqcup_{\Phi} \Phi\Big) \models \mathcal{W}.\; \qquad\Box .
  \]
\end{myco}

Since a direct
  application of this theorem incurs a prohibitive computational cost,
  it only serves to idealize the ultimate power of the framework we
  have proposed. That is, the theorem only states  that  \ALGNAME{} is
  able to select only the true patterns asymptotically (in the sample size), regardless of how the putative hypotheses size
  $\mathcal{U}$ grows, e.g., in the worst-case exponentially. It also clarifies that
  the algorithm is able to ``filter out'' all the spurious patterns (true
  negatives), and produces the true positives  more and more reliably
  as a function of the computational and data resources.

Now we restrict our attention to co-occurrence types of patterns so as
to enable a fair comparison with~\cite{beerenwinkel_2007}. We denote
with  $\mathcal{C} \subset \mathcal{U}$ the set of all possible such
patterns, and we prove the following

\begin{myth}[Inference of co-occurrence patterns]
  \label{th:sound-cnj}
  Suppose  $\Phi=\emptyset$; as before, let the sample size $m \to
  \infty$ and let the data be uniformly corrupted by false positives and
  negatives rates $\epsilon_-=\epsilon_+\in[0,1)$. Then only
  co-occurrence patterns on atomic events are inferred, which are
  either true or spurious for general CNF formulas. That is:  if
  $D(\emptyset)\models \Sigma$ then $\Sigma \subseteq \mathcal{C}$.
  Furthermore,
  \begin{enumerate}
  \item $\Sigma \cap {\mathcal{W}}$ are true patterns and 
  \item For any other pattern $\alpha \causes e  \in (\Sigma \setminus
    \Sigma \cap \mathcal{W})$  there exist $\beta \causes e\in
    \mathcal{W} \setminus \mathcal{C}$ such that $\beta$ screens off
    $\alpha$ from $e$.
  \end{enumerate}
\end{myth}
\begin{proof} Consider the proof of the previous theorem. In this case, we are dealing with formulas such that $\conj{\varphi} \subseteq G$, i.e., formulas do not have any disjunctive component. All the derivations for Theorem~2 can be carried out in this context, notice that: formulas considered in step $(i)$ of such a proof are those which are purely conjunctive and correctly inferred. Similarly, formulas in $(ii)$ are those that screen off the false patterns, but are incorrectly present in  ${\cal D}_{\text{BIC}}$.
\end{proof}

\noindent
This theorem states that, even if one is neither willing to pay the cost
of augmenting CAPRI's input with patterns  nor able to find any
suitable one, the algorithm is still capable of inferring singleton
and conjunctive instances of $\causes$ relation, whose members are
either true or part of a more complex  types of patterns that fall outside \ALGNAME{}'s scope.
An immediate corollary of these two theorems is that
\ALGNAME{} works as specified, when it is fed with all possible co-occurrence
patterns.

\begin{myco}
  \label{cor:conj-exhaustive}
  Under the hypothesis of the above theorems, $D(\emptyset)\models
  \Sigma \iff D(\mathcal{C}) \models \Sigma$. $\Box$
\end{myco}
In practice, this algorithm,  though still exponential, is certainly
less computationally intensive. For instance, when using $\mathcal{C}$ than with
$\mathcal{U}$, it can trade off computational complexity against
expressivity of the inferred patterns.


\section{Results: synthetic data} \label{sec:experiments}

\paragraph{Setting for comparison.} The performance of all the algorithms were evaluated empirically with four different types of  topologies: $(i)$  {\em trees}, $(ii)$  {\em forests}, $(iii)$ {\em DAGs without disconnected components} and $(iv)$ {\em DAGs with disconnected components}.
Irrespective of the topology considered,  we exclusively used \at{} events, which implies that either singleton or co-occurrence patterns were used in the experiments. Based on Corollary~3, it sufficed to run \ALGNAME{} with $\Phi=\emptyset$.  This strategy is consistent with the fact that our algorithm can infer more general formulas if an input ``set of putative causes, $\Phi \neq \emptyset$'' is given in addition -- a fact which, without the care taken, could have unfairly and favorably biased  our analysis in the more general situation. For the sake of completeness, however, we also tested specific CNF formulas, as shown in the next sections.

Type $(i-ii)$ topologies are DAGs constrained to have nodes with a unique parent;  condition $(i)$ further restricts such DAGs to have  no disconnected components, meaning that all nodes are reachable from a starting root $r$. Practically, condition $(i)$ satisfies  $|\pi(j)| =1$  for $j\neq r$, and $\pi(r) =\emptyset$, while in $(ii)$ we allow more roots to be present. 
This kind of topologies can be reconstructed with either ad-hoc algorithms~\cite{Loohuis:2014im,desper_1999,szabo} or general DAG-inference techniques~\cite{spirtes2000causation,tsamardinos2003algorithms,beerenwinkel_2007,bic_1978,bde_1995}. 
Type $(iii-iv)$ topologies are DAGs which have either a unique starting node $r$, or a set of independent sub-DAGs. Similarly, condition $(iii)$ satisfies  $|\pi(j)| \geq 1$  for $j\neq r$, and $\pi(r) =\emptyset$, while in $(iv)$ we allow more roots to be present, as it was in $(ii)$.  
This kind of topologies are not  reconstructible  with tree-specific algorithms, and thus only algorithms
in~\cite{spirtes2000causation,tsamardinos2003algorithms,beerenwinkel_2007,bic_1978,bde_1995} could be used for comparison. 
The algorithm for the synthetic data generation is described in the following paragraph.

\paragraph{Generating synthetic data.}  Let $n$ be the number of events we want to include in a DAG and let $p_{\min}=0.05$, $p_{\max}=0.95$, $p_{\min}=1-p_{\max}$. A {\em DAG without disconnected components} (i.e. an instance  of type $(iv)$ topology) with maximum depth $\log n$ and where each node has at most $w^\ast$ parents (i.e. $|\pi(j)| \leq w^\ast$, for $j\neq r$) is generated as follows:
\begin{algorithmic}[1]
\STATE pick an event $r\in G$ as the root of the DAG;
\STATE  assign to each  $j \neq r$ an integer  in the interval $[2, \lceil {\log n} \rceil]$ representing its depth in the DAG  ($1$ is reserved for $r$),  ensure that  each level has at least one event;
 \FORALL{events $j \neq r$} \STATE{let $l$ be the level assigned to $e$;}
 \STATE{pick $|\pi(j)|$  uniformly  over $(0,w^\ast]$, and accordingly define $\pi(j)$ with  events selected among those at which level $l-1$ was assigned;} 
 \ENDFOR
 \STATE{assign $\alpha(r)$ a random value in the interval $[p_{\min},p_{\max}]$;}
 \FORALL{events $j \neq r$} 
  \STATE{let $y$ be a random value in the interval $[p_{\min},p_{\max}]$, assign 
  \[
  \alpha(j) =  y \prod_{x \in \pi(j)} \alpha(x)\, ;
  \]} 
 \ENDFOR
\RETURN the generated DAG;
\end{algorithmic}

When an instance of type $(iv)$ topology  is to be generated, we repeat the above algorithm to create its constituent DAGs. In this case, if multiple DAGs are generated, each one with randomly sampled $n_i$ events we require that $|G|=\sum n_i=n$. 
When instances of type $(i)$ topology are required $w^\ast=1$, and by iterating multiple independent sampling instances of  type $(ii)$ topology  are generated. When required DAGs were sampled, these are  used to generate an instance of the input matrix $D$ for the reconstruction algorithms.

\subsection{Performance with different topologies and small datasets}
Here   we estimate the performance of  \ALGNAME{} for datasets with sizes that are likely to be found in currently available cancer databases, such as The Cancer Genome Atlas, TCGA~\cite{tcga}, i.e. $m \approx 250$ samples, and $15$ events.  
 The results are shown in Figure~\ref{fig:multipanel-1}, for topologies $(i)$ and $(ii)$, and Figure~\ref{fig:multipanel-2}, for  topologies $(iii)$ and $(iv)$. There, we show all the results obtained by running the algorithm with bootstrap resampling, although results (data not shown) without this pre-processing leave the conclusions unaffected.

Results suggest a  trend, as to be expected: namely, performance degrades as noise increases and sample size diminishes. However, it is particularly interesting to notice that, in various settings, \ALGNAME{} almost converges to a perfect score even with these small datasets. This happens for instance with type $(i-ii)$ topologies, where the Hamming distance almost drops to $0$ for $m\geq 150$. In general, it is also clear that reconstructing forests is easier than trees, when the same number  of events $n$ is considered. This is a consequence of the fact that, once $n$ is fixed, forests are likely to have less branches since every tree in the forest has less nodes. When reconstructing type $(iii-iv)$ topologies, instead, the convergence-speed of \ALGNAME{} to lower Hamming distance is slower, as one might reasonably expect. In fact, in those settings the distance never drops below $3$, and more samples would be required to get a perfect score. We consider this to be a remarkable result, when compared to the worst-case Hamming distance value of $15\cdot14=210$.
Panels of Figure~\ref{fig:multipanel-2} also suggest that disconnected DAGs are easier to reconstruct than connected ones, when a fixed number  of events is considered. Similarly to the above, this could be credited to the fact that the size of the conjunctive claims is generally smaller, for fixed $n$. With respect to the precision and recall scores, one may note that \ALGNAME{} seems to be quite robust to noise, since the loss in the score-values appear nearly unaffected by any increase in the noise parameter.

\subsection{Comparison with other reconstruction techniques}
We compare now with state-of-the-art approaches mentioned in the main text\footnote{Classic versions of the  IAMB and PC algorithm were further subjected to log-likelihood optimization to assign a direction to all of the computed non-oriented edges.  This additional feature is necessary to permit a fair comparison against various structural approaches, which, otherwise, would be penalized with a worse Hamming distance, since these algorithms, in principle, can return non-oriented edges. Note that progression models, by their very nature, consist only of oriented structures.}, which we divide into three categories: {\em structural} - Incremental Association
  Markov Blanket (IAMB) and PC algorithm -, {\em likelihood} - Bayesian
  Information Criterion (BIC) and Bayesian
  Dirichlet (BDE) and {\em hybrid}  - Conjunctive Bayesian Networks (CBN) and Cancer Progression Inference  with Single Edges (CAPRESE).  For all the algorithms we used their standard  {\sc r} implementations: for IAMB, BDE and BIC we used  package \texttt{bnlearn}~\cite{bnlearn}, for the PC algorithm we used package \texttt{pcalg},  for \ALGNAMEPLOS{} we used  \texttt{TRONCO}~\cite{tronco} (first release) and for CBN we used \texttt{h-cbn}~\cite{h-cbn}.

Clearly, other algorithms exist in the literature, but we selected those which satisfied at least one of the following criteria: earlier, they have proven to be more effective in inferring ``causal'' claims, i.e., they are considered the best algorithms to infer ``causal networks'' (i.e., IAMB and PC); they regularize the Bayesian over-fit (i.e., BDE and BIC); they assume a prior (i.e. BDE) or they were developed specifically for cancer progression inference (i.e., CBN and \ALGNAMEPLOS{}). 
Prominent among the ones absent in this study are the following: {\em Grow and Shrink}~\cite{gs}, which preliminary analysis have shown to be very similar to IAMB, and the {\em DiProg algorithm}~\cite{farahani}, which unrealistically requires advanced knowledge of input error rate to reconstruct a model; note that this kind of information is not generally available {\em a priori}.

Notice that we selected all the algorithms capable of inferring generic DAGs but \ALGNAMEPLOS{}~\cite{Loohuis:2014im}, which can only be applied to infer trees or forests (i.e., type $(i-ii)$ topologies). In the literature there exist other approaches specifically tailored for such topologies, e.g.,~\cite{desper_1999, szabo}; however, since  in~\cite{Loohuis:2014im} it is shown that \ALGNAMEPLOS{}  performs better than other approaches, we assume no loss of information in restricting our study. We place  \ALGNAME{} in the {\em Hybrid} category, though we clearly compare its performance with all the other approaches in order to quantify its suitability for reconstruction of all classes of topologies, as defined earlier.

The general trend is summarized in Figure~\ref{fig:comparison}, where we rank all of these algorithms  according to their median performance, estimated as a function of noise and sample size, and provide the parameters  used for comparison.  In Figure~\ref{fig:structural}, we compare \ALGNAME{}  with the structural approaches (IAMB and PC).  In Figure~\ref{fig:likelihood}, we compare it with the likelihood approaches (BIC and BDE) and, finally, in   Figure~\ref{fig:hybrid}, we compare it with the hybrid algorithms. We remark that,  because of the high computational cost of running CBNs, which relies on a nested Expectation-Maximization algorithm with Simulated Annealing, the number of ensembles performed is limited to $100$ for CBNs, while it is $1000$ for all other algorithms. Though this strategy provides less robust statistics for CBNs (i.e., less ``smooth'' performance surfaces), it is still sufficiently accurate to indicate the general comparative trends and relative performance efficiency.

\subsection{Reconstruction without hypotheses: disjunctive patterns}

Recall that our algorithm expects as input all the hypothesized patterns to infer more expressive logical formulas, i.e., hypotheses with pure CNF formulas or even disjunctive patterns over atomic events. Nonetheless, it is instructive to investigate its performance under two specific conditions, especially to clarify the robustness with respect to imperfect regularities (the, e.g, ``noisy and''): namely, $(i)$ without hypotheses ($\Phi=\emptyset$) and $(ii)$ for datasets sampled from topologies with {\em disjunctive} patterns.

To generate the input dataset, we have to modify the generative procedure used for the other tests, thus reflecting the switch from co-occurrence to disjunctive patterns. This task is actually rather simple, since we just change the labeling function $\alpha$ to account for the probability of picking any subset of the clauses in the disjunctive pattern, while omitting the others.  We use DAGs with  $10$ events and disjunctive patterns with at most $3$ atomic events involved, which is a reasonable size, given the events considered. Clearly,   this setting is generally harder  than the one shown in Figures~\ref{fig:structural}--~\ref{fig:hybrid}, thus we expect performance to be somewhat inferior.
  Here we compare \ALGNAME{} with all the algorithms used so far, and we show  the result of  this comparison in Figure~\ref{fig:disjunctive-DAG-bestof}, where $\Phi=\emptyset$, as noted earlier. The plot clearly confirms the trends suggested by previous analyses: namely, \ALGNAME{} infers the correct patterns more often than the others. {Note also that the performance is measured on the reconstructed topology only, since, without input hypotheses, the algorithm evaluates only co-occurrence types of patterns, and does not allow different types of relations (e.g. disjunctions) to be inferred automatically.} However, as anticipated, observed performance improvement is now much lower, and the Hamming distance fails to rise above $4$. Furthermore, convergence to optimal performance  was not observed for $m \leq 1000$, and  it appears not to be reachable even for $m \gg 1000$ (at least so, when no  hypotheses are used). It is also possible that, as $n$ and the number of maximum disjunctive patterns increase, the result could be an even less satisfactory speed of convergence.

\subsection{Reconstruction with hypotheses: synthetic lethality}

We wondered whether \ALGNAME{} would be able to infer synthetic lethality relations, when these are directly hypothesized in the input set $\Phi$. We started with a test of the simplest form: e.g., 
$
[a \oplus b \;\causes\; c],
$
for a set of events $G=\{a,b,c\}$, where we force progression from $a$ to $c$ to be preferential, i.e. it appears with $0.7$ probability, whereas $b$ to $c$ does so with only $0.3$ probability, thus implying that samples involving $(a \wedge \bar{b})$ 
will be more abundant than those involving $(\bar{a} \wedge b)$. 
Despite this being the smallest possible synthetically lethal pattern, the goal was to estimate the {\em probability} of such a pattern being robustly inferable, when $\Phi=\{a \oplus b \causes c\}$, and its dependence on the sample size and noise. We measured the performance of all the algorithms, with an input lifted according to the pattern so that all algorithms start with the same initial pieces of information. The performance metric estimates how likely an edge from $a \oplus b$ to $c$ could be found in the reconstructed structures.

We show the results of this comparison  in Figure~\ref{fig:multipanel-lethality}. We note that  \ALGNAME{} succeeds in inferring the synthetic lethality relation more frequently than $93\%$ of the times, irrespective of the noise and sample size used. More precisely, with $m\geq60$ the algorithm infers the correct pattern under any execution, thus suggesting that \ALGNAME{}, with the correct input hypotheses, is able to infer complicated structures, many of which could have high biological significance. Naturally, it would be reasonably expected that the performance of any of these algorithms would drop, were the target relations part of a bigger model. 

\subsection{Execution time}
We report an evaluation of the execution time for all the algorithms we tested, but CBN - which computation time is more than one order of magnitude higher than the competing techniques. Two distinct settings of experiments were used: Setting $(A)$: $n=10$ events, $m=100$ samples, $\nu=0$ noise; Setting $(B)$: $n=10$, $m=100$, $\nu=.10$. Results account for the average time of execution as of $100$ randomly generated topologies (one dataset sampled per topology). Time unit is {\em second} and the test was performed on a MacBook with 
$2.3$ GHz Intel i7 processor, $16$ Gb of RAM and Yosemite 10.9 OS.

To allow a fair comparison of CAPRI against the other algorithms we both executed the algorithm with and without bootstrap preprocessing, in order to asses the prima facie condition (Mann-Withney U test being performed in the former case). Execution timings are sorted according to mean time.

\begin{table}
\begin{center}
\begin{tabular}{c|c|c|c}
{\bf{Setting A} ($\nu = 0$)} & mean & median & standard deviation \\
 \hline
\ALGNAMEPLOS{} & $0.006$ & $0.005$ & $0.005$ \\
 \hline
BIC & $0.023$ & $0.022$ & $0.011$ \\
 \hline
IAMB  & $0.028$ & $0.027$ & $0.005$ \\
 \hline
\ALGNAME{} without bootstrap & $0.029$ & $0.029$ & $0.003$ \\
 \hline
BDE  & $0.041$ & $0.032$ & $0.063$ \\
 \hline
PC & $0.144$ & $0.112$ & $0.154$ \\
 \hline
\ALGNAME{} with bootstrap & $1.143$ & $1.056$ & $0.360$ \\
 \hline
 \end{tabular}
\end{center}

\begin{center}
\begin{tabular}{c|c|c|c}
{\bf{Setting B} ($\nu = .10$)}  & mean & median & standard deviation \\
 \hline
\ALGNAMEPLOS{} & $0.005$ & $0.005$ & $0.001$ \\
 \hline
BIC & $0.022$ & $0.022$ & $0.003$ \\
 \hline
IAMB  & $0.029$ & $0.028$ & $0.004$ \\
 \hline
\ALGNAME{} without bootstrap & $0.030$ & $0.029$ & $0.004$ \\
 \hline
BDE  & $0.030$ & $0.028$ & $0.010$ \\
 \hline
PC & $0.103$ & $0.094$ & $0.034$ \\
 \hline
\ALGNAME{} with bootstrap & $0.719$ & $0.689$ & $0.138$ \\
 \end{tabular}
\end{center}
\end{table}

\section{Biological examples}

\subsection{Atypical Chronic Myeloid Leukemia}
\subsubsection*{Input hypotheses for CAPRI (supervised mode)}
By fetching the literature we selected the following patterns to input as \ALGNAME{}'s hypotheses:
\begin{itemize}
\item[$(1)$] ``exclusivity among \asxl{} and \sfb{} mutations'' \cite{lin2014sf3b1}: 
\begin{center}
(\asxl{} {\em Nonsense point} $\oplus$ \asxl{} {\em Ins/del} )  $\oplus$ \sfb{} {\em Missense point}
\end{center}

\item[$(1)$] ``exclusivity among \tet{} and \idh{} mutations'' \cite{figueroa2010leukemic}: 
\begin{center}
(\tet{} {\em Nonsense point} $\oplus$ \tet{} {\em Missense point} $\oplus$ \tet{} {\em Ins/del}  )  $\oplus$ \idh{} {\em Missense point}
\end{center}

\end{itemize}
These patterns were used to build CAPRI's hypotheses which were tested against all  events  which do not appear in  the above pattern itself, e.g., pattern (1) was tested against all  input events but those involving \asxl{} and \sfb{} genes.

As shown in the main text, among all, the following hypothesis gets selected by CAPRI 

\begin{center}
(\asxl{} {\em Nonsense point} $\oplus$ \asxl{} {\em Ins/del} )  $\oplus$ \sfb{} {\em Missense point} $\causes$ \cbl{} {\em Missense point}
\end{center}

\subsubsection*{aCML progression model with different techniques}
In Figure \ref{fig:results_real_acml} one can find the progression models reconstructed on the the \aCML{} dataset \cite{piazza_nat_gen}, with 3 different algorithms: $(i)$ CAPRESE, $(ii)$ BIC and $(iii)$ IAMB. These three techniques were chosen for this comparative study because of the overall better performance on synthetic tests (see Section 3.2-3.4 of the SI). The reconstruction obtained with CAPRI can be found in Figure 5 in the main text. 
For a biological interpretation of the results please refer to Section 4.2 in the main text. 

Note that all the progression model share some specific selective advantage relations, yet being substantially different. 
Relations involving \setbp{} and \asxl{} and those involving \tet{} and \ezh{} are, in fact, inferred by all the four algorithms, yet with different confidences and, sometimes, edge direction. 
In addition, IAMB does not include \cbl{} in the path involving \setbp{} and \asxl{}, and none of the algorithms but CAPRI can infer the complex pattern involving \asxl{} mutations of both types and \sfb{} (Figure 5 in the main text). 
Finally, note that IAMB and BIC are often not able to disambiguate the edge direction and this represent a major limit of these techniques with respect to CAPRI and CAPRESE. 

\subsection{Ovarian cancer}
\subsubsection*{Ovarian cancer progression model with different techniques}
We analyzed an ovarian cancer dataset  reporting chromosome-level
amplifications and deletions detected via Comparative Genome
Hybridization in \cite{knutsen2005interactive}. Similar to the case of
aCML, we used 4 different techniques to infer a progression models for
events included in the dataset:  CAPRI (unsupervised), CAPRESE,  BIC
and IAMB. Models and input dataset are shown in Figure \ref{fig:results_real_ovarian}. 
Like with aCML extraction, the progression models share only some of the inferred
relations. Among the most relevant differences is the conjunctive
pattern inferred  by CAPRI between the loss on chromosome \emph{5q}
($5q-$) and the gain on chromosome \emph{8q} ($8q+$) which is
predicted to select for a loss on $8p$; note also the aforementioned
limitation of BIC and IAMB in disambiguating the direction of some of
the inferred relations. Note that CAPRI infers a co-occurrence
pattern of selective advantage which is not input a priori as
hypothesis - unsupervised execution. In summary, CAPRI displays a better overall confidence on the reconstructed model. 

\newgeometry{top=1cm,left=1cm,bottom=1cm,right=1cm}
\begin{figure}[p]
\begin{center}
\centerline{\includegraphics[width=1.0\textwidth]{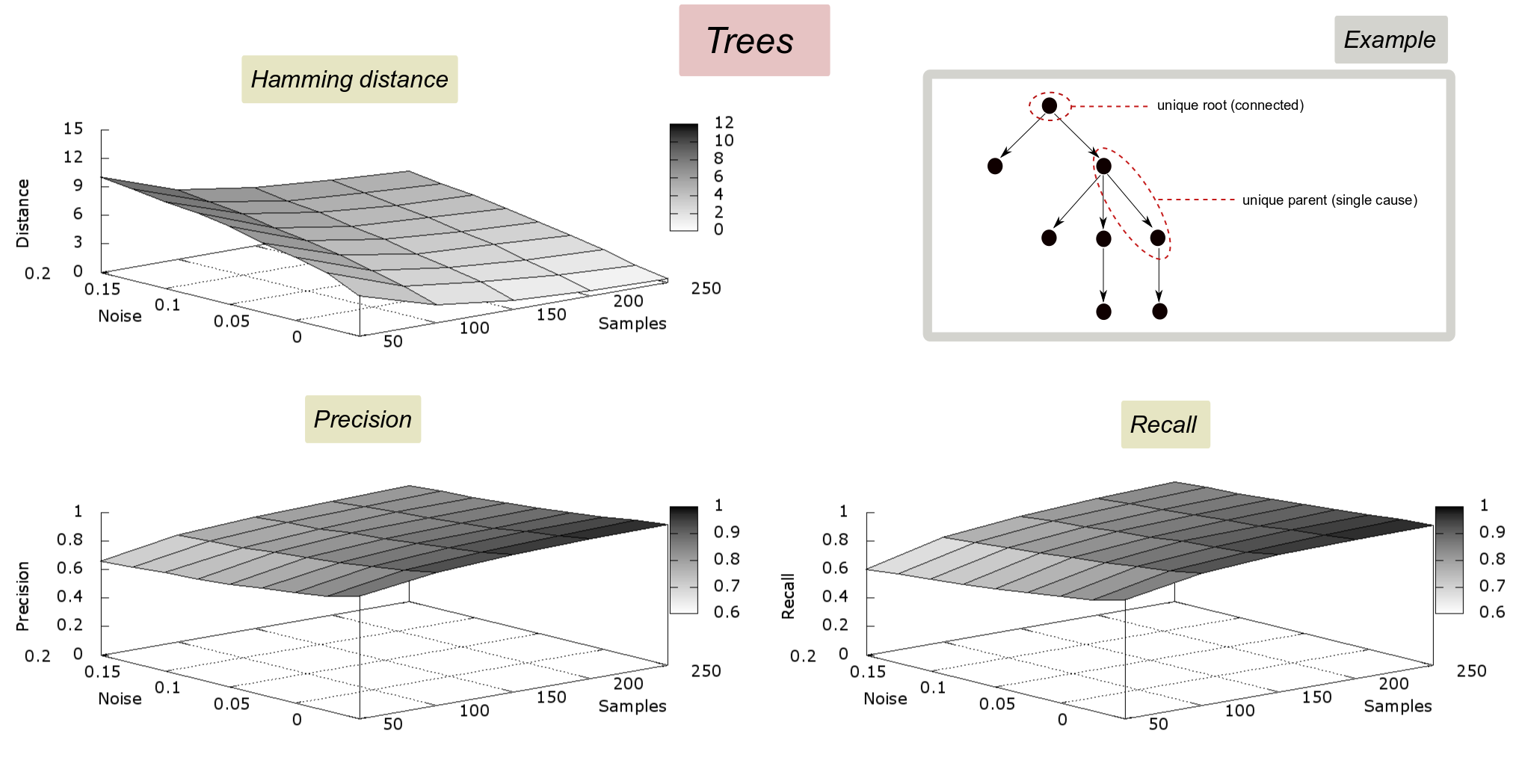}}
\centerline{\includegraphics[width=1.0\textwidth]{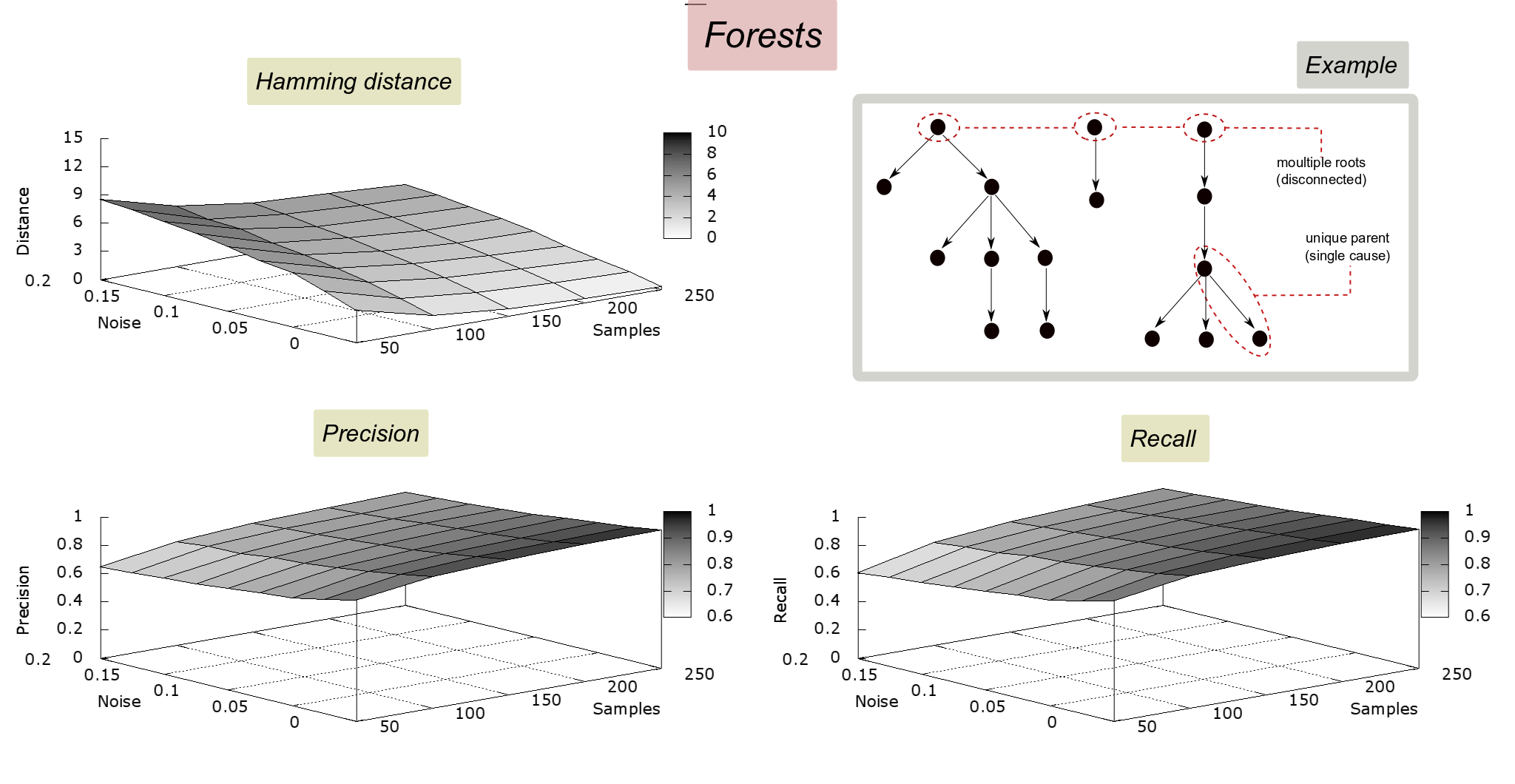}}
\end{center}
\caption[Reconstruction of trees and forests with small datasets]{{\bf Reconstruction of trees and forests with small datasets.} Hamming distance, precision and recall of \ALGNAME{} for synthetic data generated by trees (i.e.,  models with a singleton pattern per event and a unique progression), in top panels, and by  forests (i.e.,  models with a singleton pattern per event but  multiple independent progressions), in bottom panels. In both cases  $n=15$ events are considered, $m$ ranges from $50$ to $250$ and the noise rate ranges from $0\%$ to $20\%$. To have a reliable statistics, for each type of topology, we generate $100$ distinct progression models and, for each value of sample size and noise rate, we sample $10$ datasets from each topology. Thus, every performance entry  is the average of $1000$ reconstruction results.  Notice that Hamming distance  almost drops to $0$ for $m\geq 150$ and that precision and recall decrease very little  as noise increases.}
\label{fig:multipanel-1}
\end{figure}
\thispagestyle{empty}
\restoregeometry

\newgeometry{top=1cm,left=1cm,bottom=1cm,right=1cm}
\begin{figure}[p]
\begin{center}
\centerline{\includegraphics[width=1.0\textwidth]{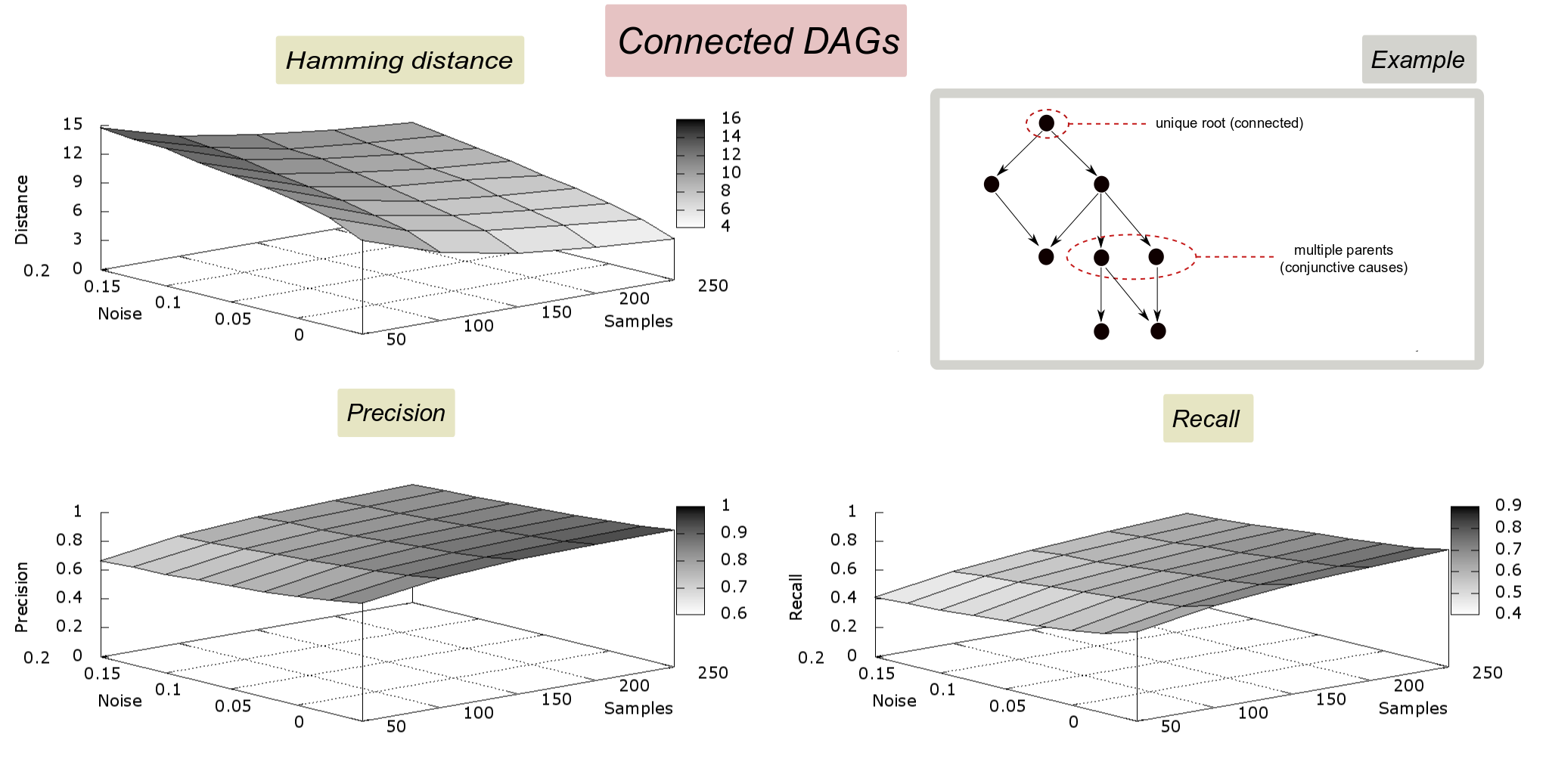}}
\centerline{\includegraphics[width=1.0\textwidth]{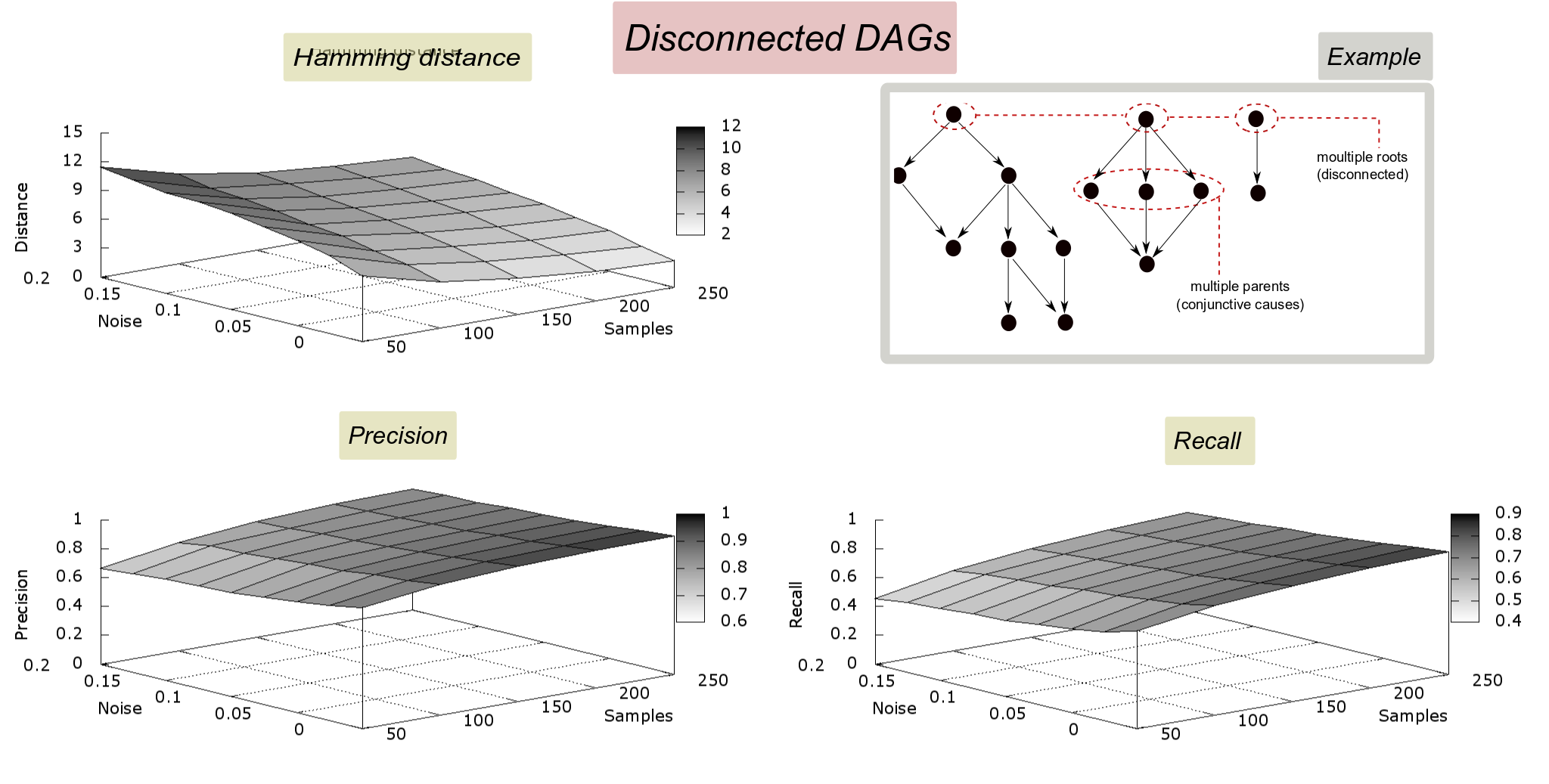}}
\end{center}
\caption[Reconstruction of DAGs with small datasets]{{\bf Reconstruction of DAGs with small datasets.} Hamming distance, precision and recall of \ALGNAME{} for synthetic data generated by connected DAGs  (i.e.,  models with either a singleton or co-occurrence pattern per event and a unique progression), in top panels, and by  disconnected DAGs (i.e.,  models with either a singleton or co-occurrence pattern per event and multiple progressions), in bottom panels. In both cases the same parameters as in Figure~\ref{fig:multipanel-1} are used ($n=15$, $50 \leq m\leq 250$, $0\% \leq \nu \leq 20\%$ and every performance entry  is the average of $1000$ reconstructions). In this setting, which is harder than the one shown in Figure~\ref{fig:multipanel-1},  Hamming distance  does not reach values below $3$ -- a reasonably small number for our purposes -- while  precision and recall still suffer very little  as noise increases. }
\label{fig:multipanel-2}
\end{figure}
\thispagestyle{empty}
\restoregeometry

\newgeometry{top=.2cm,left=1cm,bottom=1cm,right=1cm}
\begin{figure}[p]
\begin{center}
  \centerline{\includegraphics[width=0.68\textwidth]{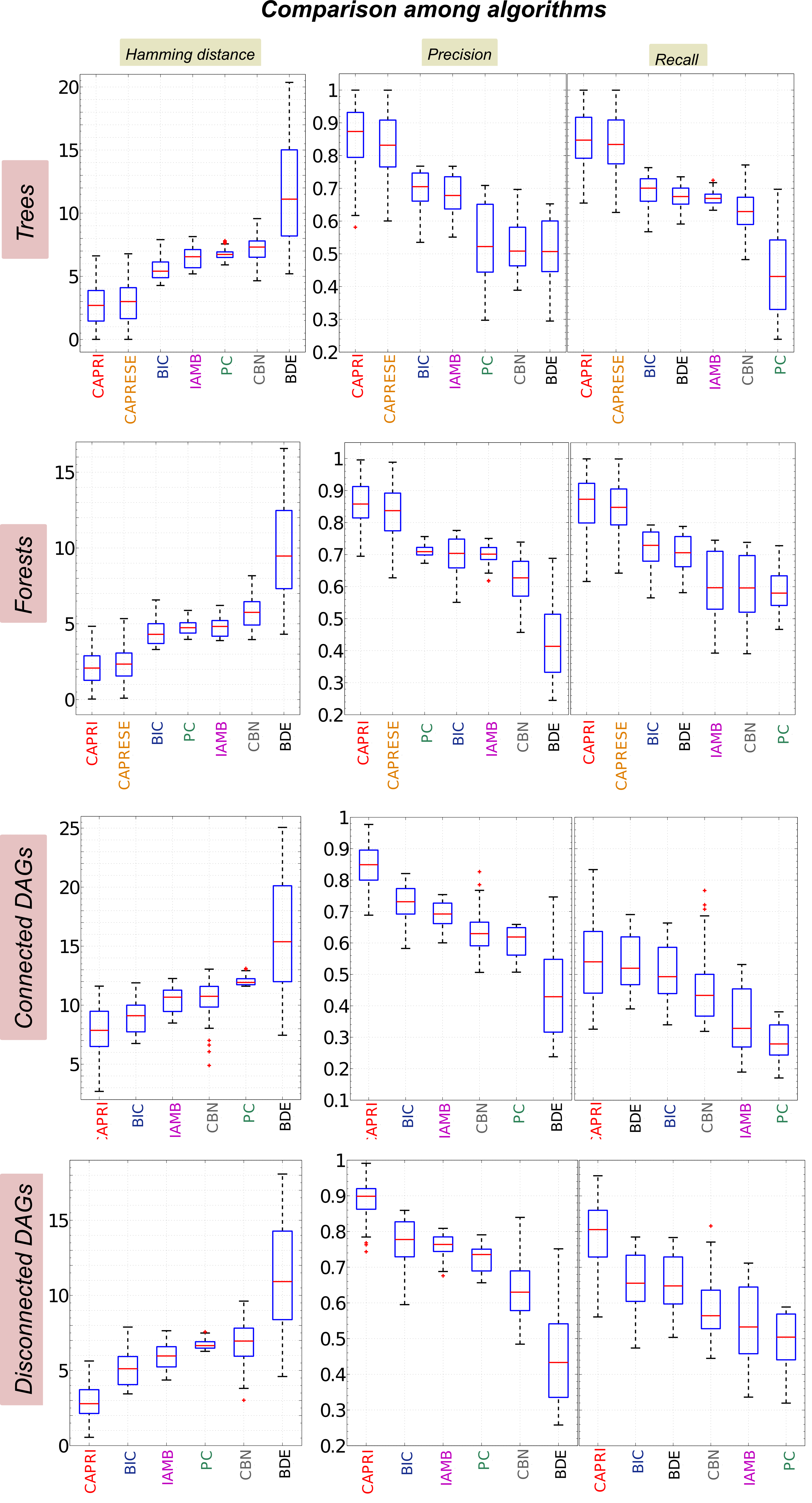}}
  \vspace{1.0cm}
     \begin{tabular}{c |  l | l }
\multicolumn{3}{c}{{\bf Parameter values}}  \\ \hline
$n$ & {\em number of events} &  $10$ \\
$m$ & {\em number of samples  } &  $ [50, 1000]$ \\
$\nu$ & {\em rate of false positives $\epsilon_+$ and negatives $\epsilon_-$} & $[0, 0.2]$ ($0\%$-$20\%$ noise rate)\\
$-$ & {\em ensemble size} & $1000$ ($100$ for CBN)
  \end{tabular}
\end{center}
\caption[Co-occurrence patterns: performance ranking]{{\bf Co-occurrence patterns: performance ranking.} We rank 
the algorithms we compared in Figure~\ref{fig:structural},~\ref{fig:likelihood}  and~\ref{fig:hybrid} according to their performance for the parameters in the  table. Rankings are divided according to the topology type and sorted according to the median performance.  
}
\label{fig:comparison}
\end{figure}
\thispagestyle{empty}
\restoregeometry

\newgeometry{top=.2cm,left=1cm,bottom=1cm,right=1cm}
\begin{figure}[p]
\begin{center}
\centerline{\includegraphics[width=0.88\textwidth]{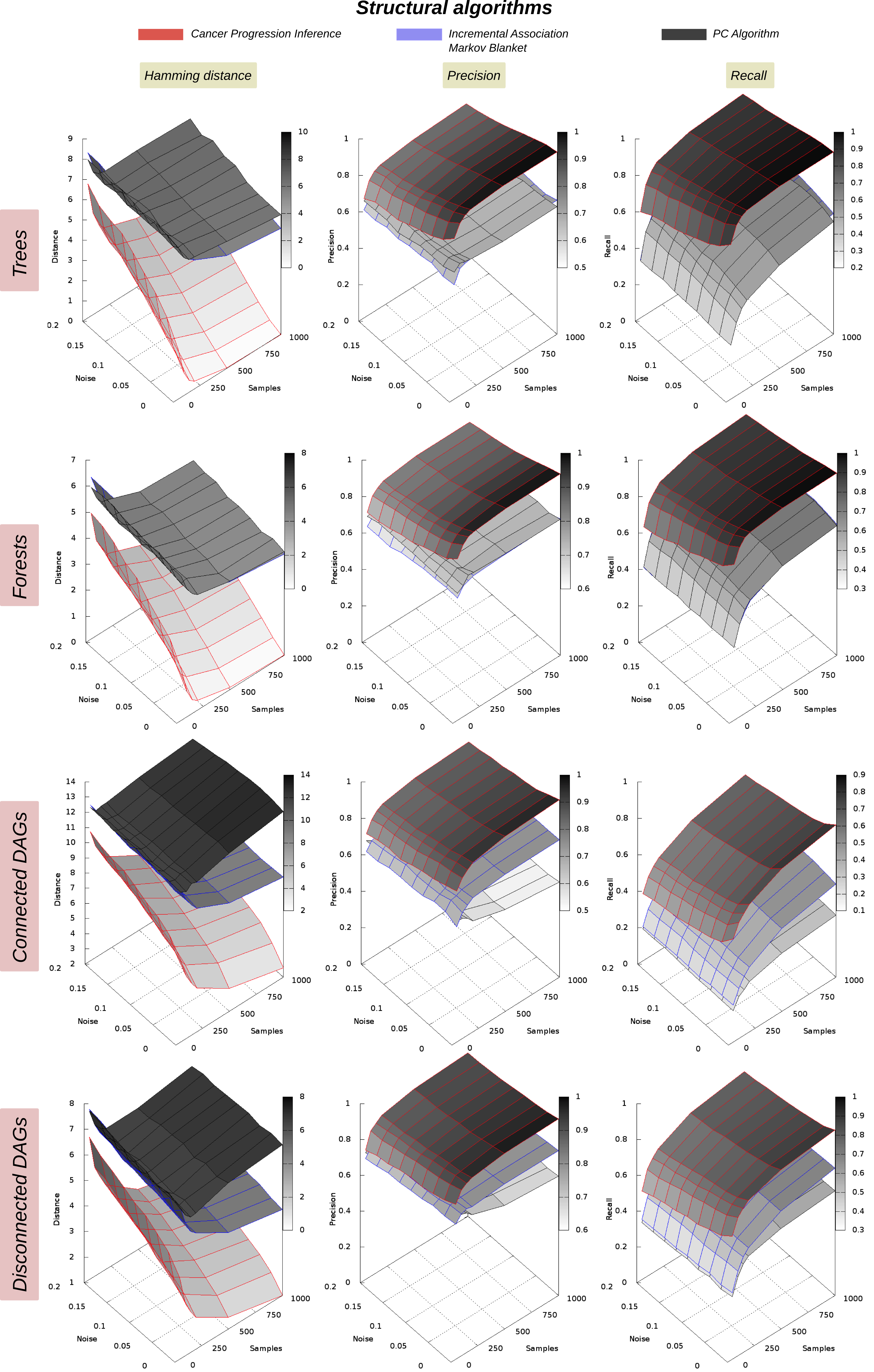}}
\end{center}
\caption[Comparison with related works: structural algorithms]{{\bf Comparison with related works: structural algorithms.} We compare  \ALGNAME{},  IAMB  and    the  PC algorithm  to infer {\em trees}, {\em forests}, {\em connected DAGs} and {\em disconnected DAGs} with  the parameters described in Table~\ref{fig:comparison}. Average Hamming distance, precision and recall are shown.}
\label{fig:structural}
\end{figure}
\thispagestyle{empty}
\restoregeometry

\newgeometry{top=.2cm,left=1cm,bottom=1cm,right=1cm}
\begin{figure}[p]
\begin{center}
\centerline{\includegraphics[width=0.88\textwidth]{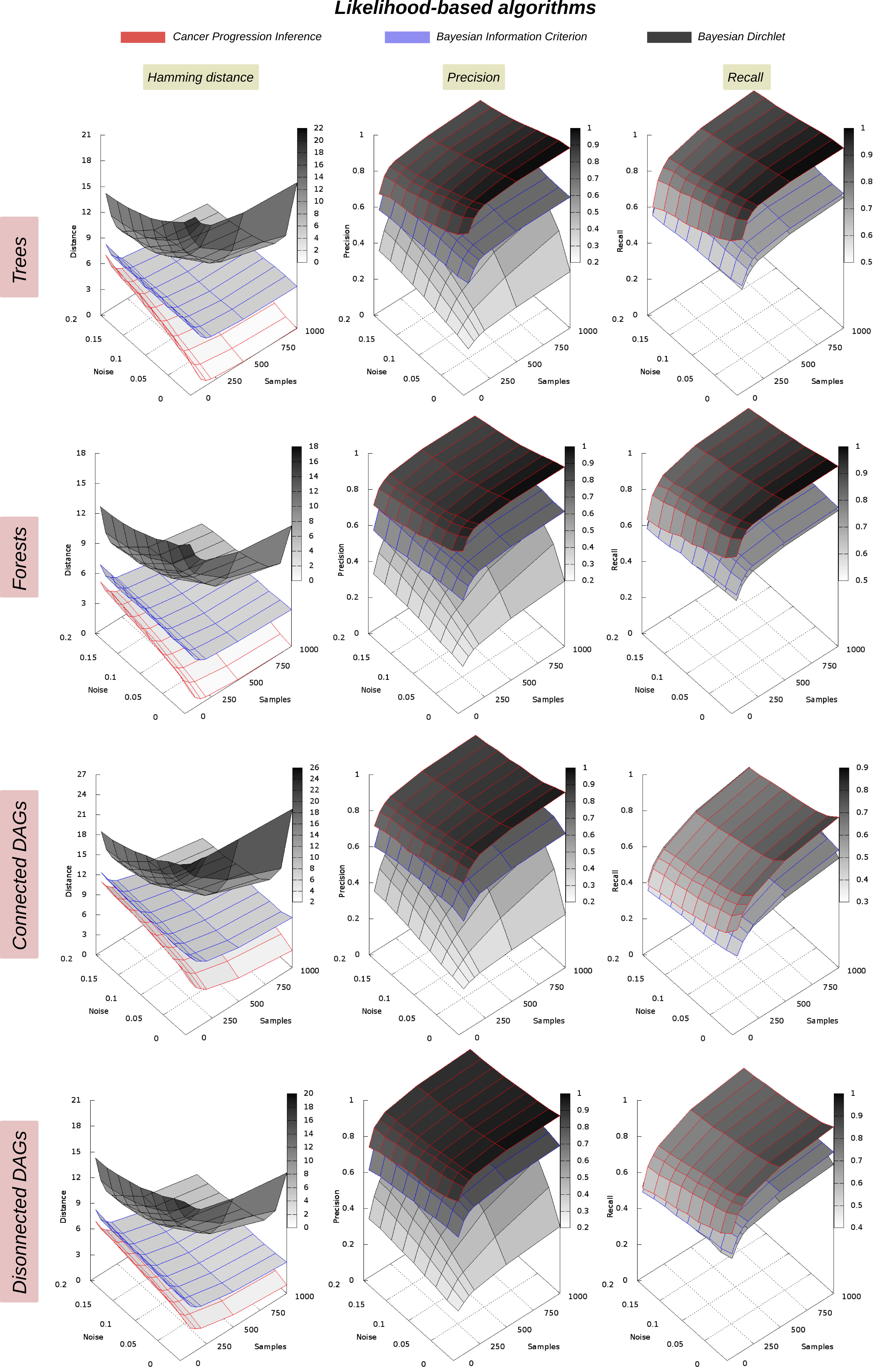}}
\end{center}
\caption[Comparison with related works: likelihood-based algorithms]{{\bf Comparison with related works: likelihood-based algorithms.} We compare  \ALGNAME{} against likelihood-based methods optimizing BIC and BDE scores to infer {\em trees}, {\em forests}, {\em connected DAGs} and {\em disconnected DAGs} with  the parameters described in Table~\ref{fig:comparison}.  Average Hamming distance, precision and recall are shown.}
\label{fig:likelihood}
\end{figure}
\thispagestyle{empty}
\restoregeometry

\newgeometry{top=.2cm,left=1cm,bottom=1cm,right=1cm}
\begin{figure}[p]
\begin{center}
\centerline{\includegraphics[width=0.88\textwidth]{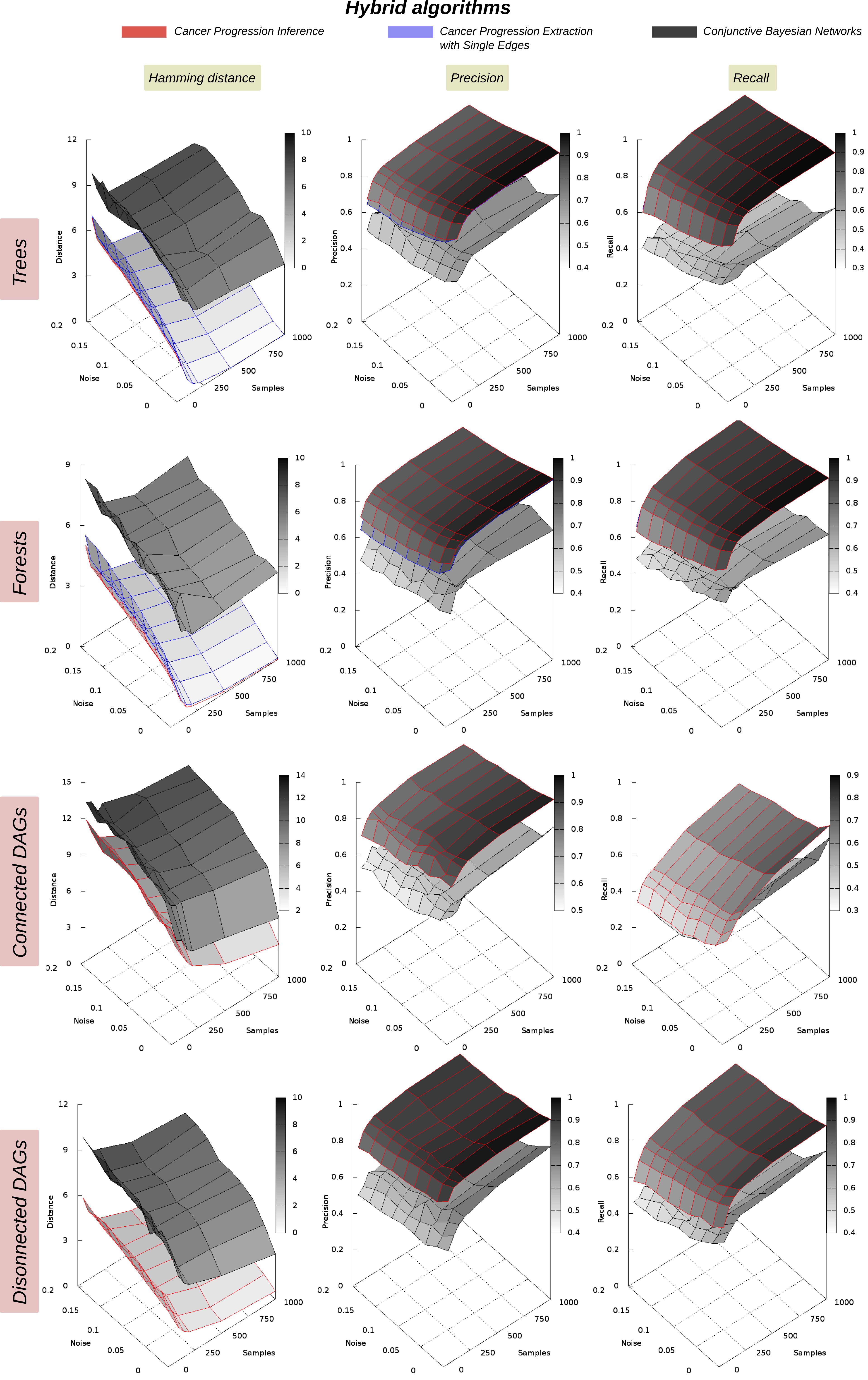}}
\end{center}
\caption[Comparison with related works: hybrid algorithms]{{\bf Comparison with related works: hybrid algorithms.} We compare  \ALGNAME{},  CBNs  and \ALGNAMEPLOS{} to infer {\em trees}, {\em forests}, {\em connected} and {\em disconnected DAGs} with  the parameters of Table~\ref{fig:comparison}   but, because of the computational cost of running CBNs with $100$ annealing steps, we reduced the number of ensembles performed as: $100$ for CBNs,  $1000$ for \ALGNAMEPLOS{}  and, for \ALGNAME{}, $100$ for DAGs and $1000$ otherwise. Average Hamming distance, precision and recall are shown.}
\label{fig:hybrid}
\end{figure}
\thispagestyle{empty}
\restoregeometry

\newgeometry{top=1cm,left=1cm,bottom=1cm,right=1cm}
\begin{figure}[p]
\begin{center}
\centerline{\includegraphics[width=1.0\textwidth]{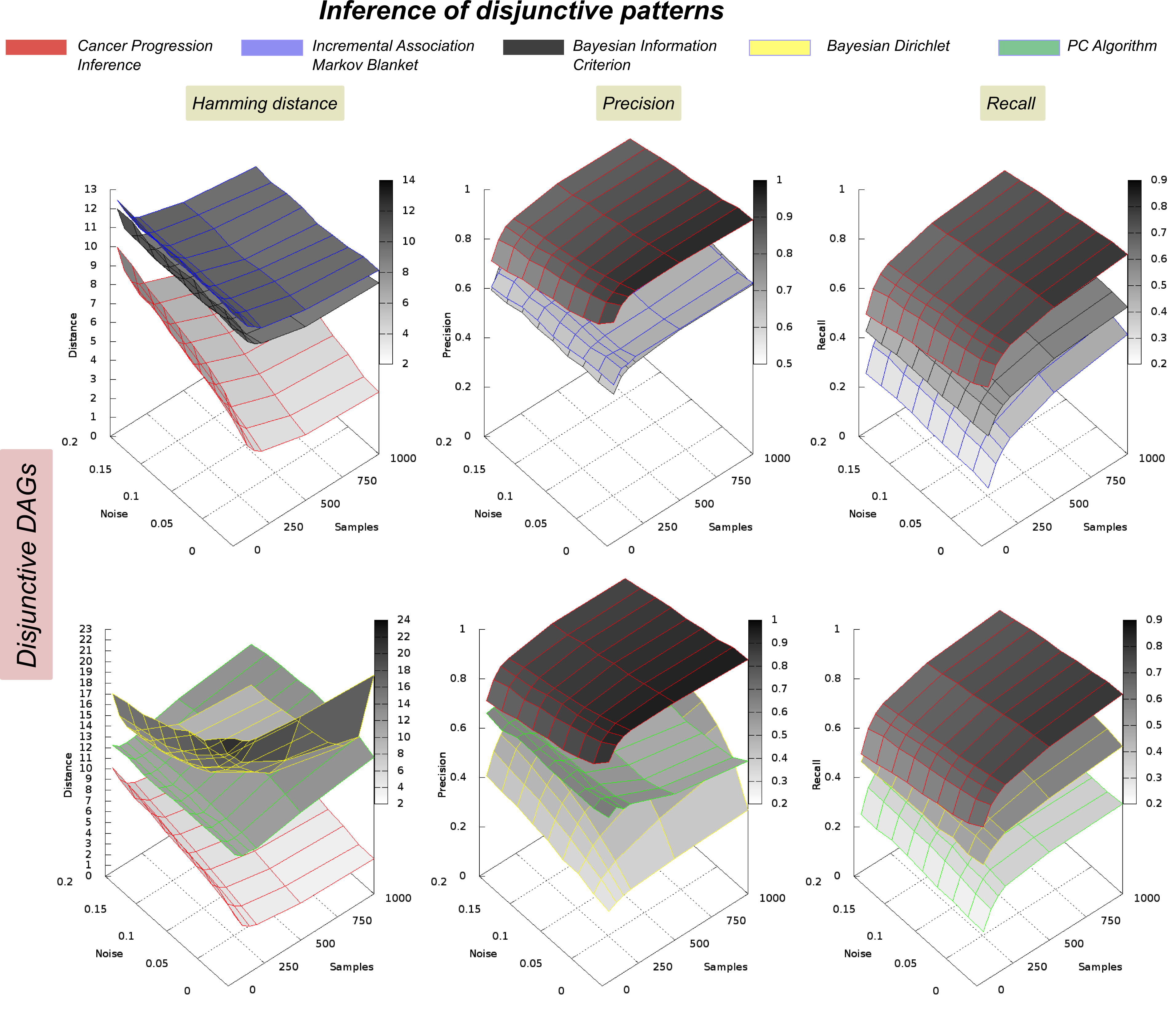}}
\end{center}
\caption[Reconstruction of disjunctive patterns with  no hypotheses]{{\bf Reconstruction of disjunctive patterns with  no hypotheses.}  We compare \ALGNAME{} against all the algorithms to infer progressions with disjunctive patterns. In top panel we show IAMB as the best structural algorithm, and the BIC score as
the best  among likelihood-based methods, according to Table~\ref{fig:comparison}. In bottom panel we compare the other algorithms. No hypotheses ($\Phi=\emptyset$) are given as input to \ALGNAME{}. Input data is  generated by  DAGs with  $10$ atomic events and disjunctive patterns with at most $3$ atomic events involved. Sample size ranges from $50$ to $1000$, noise rate from $0\%$ to $20\%$ and $1000$ ensembles are generated for each configuration of noise and sample size. This setting is generally harder  than the one shown in Figures~\ref{fig:structural}--~\ref{fig:hybrid}. Hamming distance, precision and recall are shown and confirm that this type or pattern is  harder than the co-occurent one to be inferred, hinting at the difficulty of modeling unbalanced confluent progressions.
}
\label{fig:disjunctive-DAG-bestof}
\end{figure}
\thispagestyle{empty}
\restoregeometry

\newgeometry{top=1cm,left=1cm,bottom=1cm,right=1cm}
\begin{figure}[p]
\begin{center}
\centerline{\includegraphics[width=1.0\textwidth]{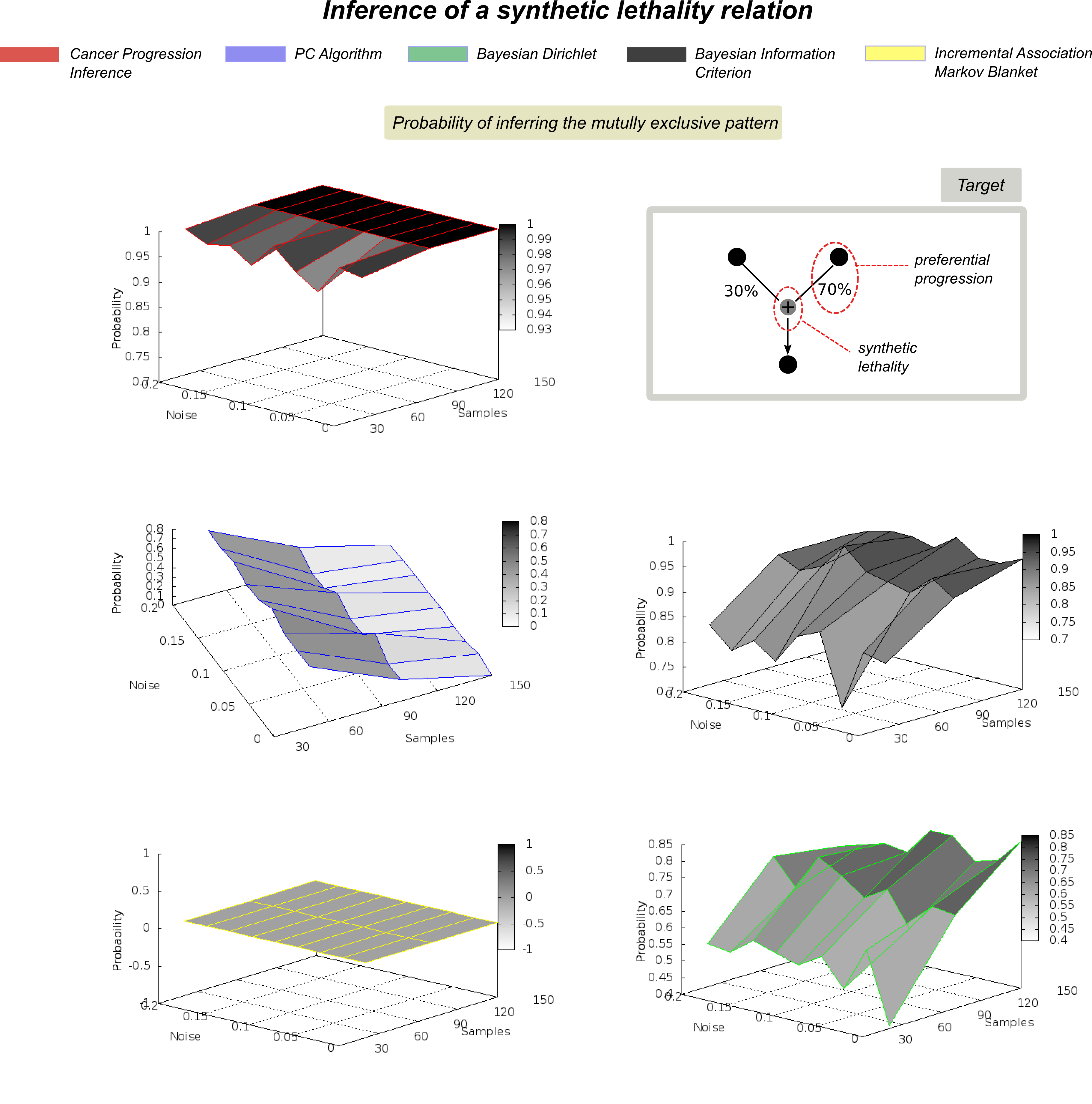}}
\end{center}
\caption[Reconstruction with hypotheses: synthetic lethality]{{\bf Reconstruction with hypotheses: synthetic lethality.} 
We show the  {\em average probability} of inferring a claim $a \oplus b \causes c$ ({\em synthetic lethality}), when this is provided in the input set $\Phi$.  We show such a probability for \ALGNAME{}, the likelihood-based algorithms with BIC and BDE scores, and the structural IAMB and PC Algorithm. Data is generated from the model in the upper left panel (unbalanced ``exclusive or'' with a preferential progression), samples size ranges from $30$ to $120$, noise rate  from $0\%$ to $20\%$ and $1000$ ensembles are generated for each configuration of noise and sample size.
Results suggest that a threshold level on the number of samples exists such that \ALGNAME{} infers the correct claim when $\Phi=\{a \oplus b \causes c\}$. We executed  all the algorithms with an input matrix lifted to contain the target claim. }
\label{fig:multipanel-lethality}
\end{figure}
\thispagestyle{empty}
\restoregeometry

\newgeometry{top=.2cm,left=1cm,bottom=1cm,right=1cm}
\begin{figure}[p]
\begin{center}
\centerline{\includegraphics[width=1.0\textwidth]{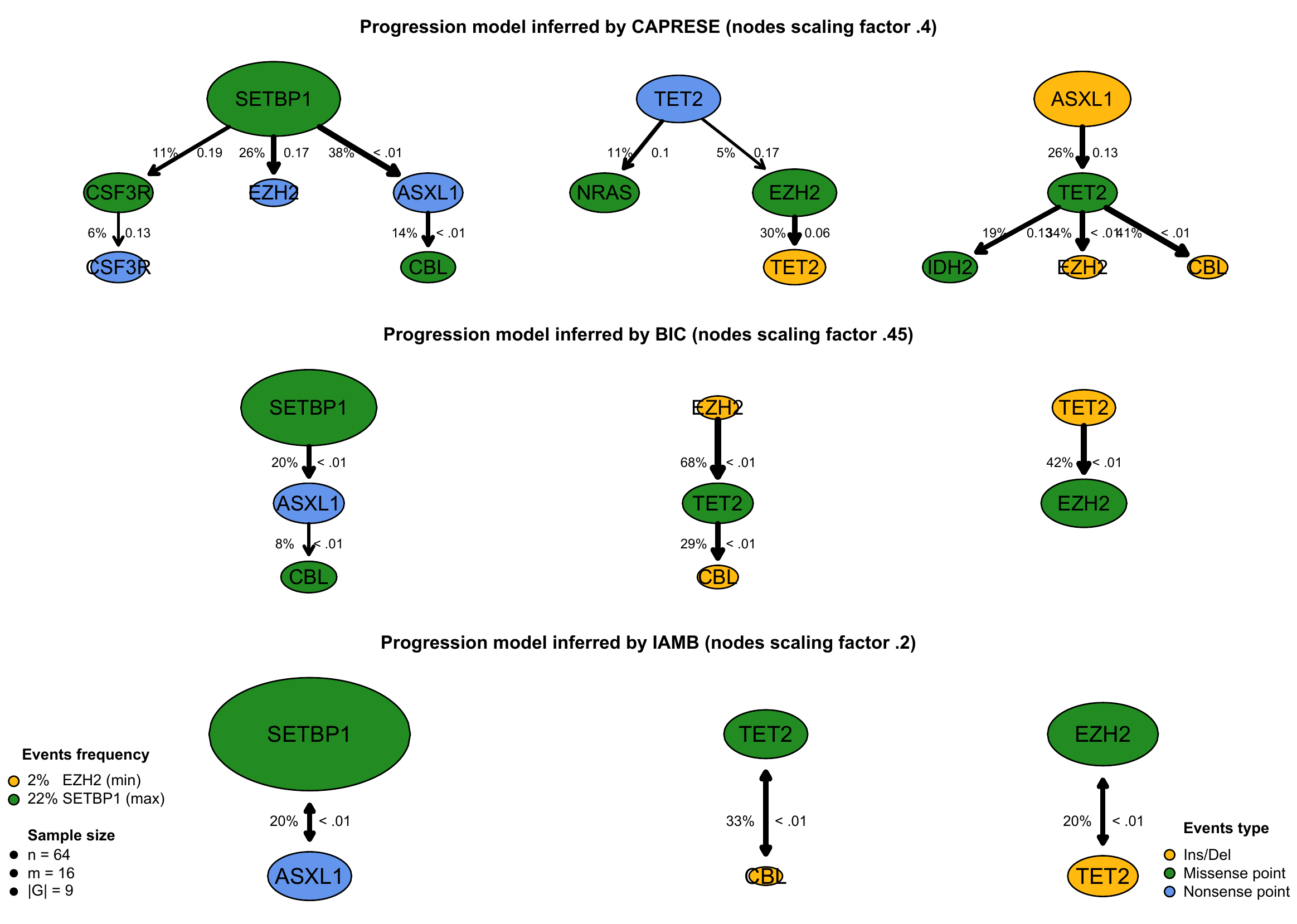} }
\end{center}
\caption[CAPRESE, IAMB and BIC progression models of \aCML{}]{{\bf CAPRESE, IAMB and BIC progression models of \aCML{}.} Progression models reconstructed from the \aCML{} dataset described in the main text - taken from \cite{piazza_nat_gen} - obtained with the following algorithms:  CAPRESE,  IAMB and  BIC. The model inferred by CAPRI is shown in the Main Text. Confidence shown is assessed as  for the CAPRI algorithm. Nodes are scaled differently to better layout the graphs reconstructed by every algorithm.}
\label{fig:results_real_acml}
\end{figure}
\thispagestyle{empty}
\restoregeometry

\newgeometry{top=.2cm,left=1cm,bottom=1cm,right=1cm}
\begin{figure}[p]
\begin{center}
{\includegraphics[width=0.65\textwidth]{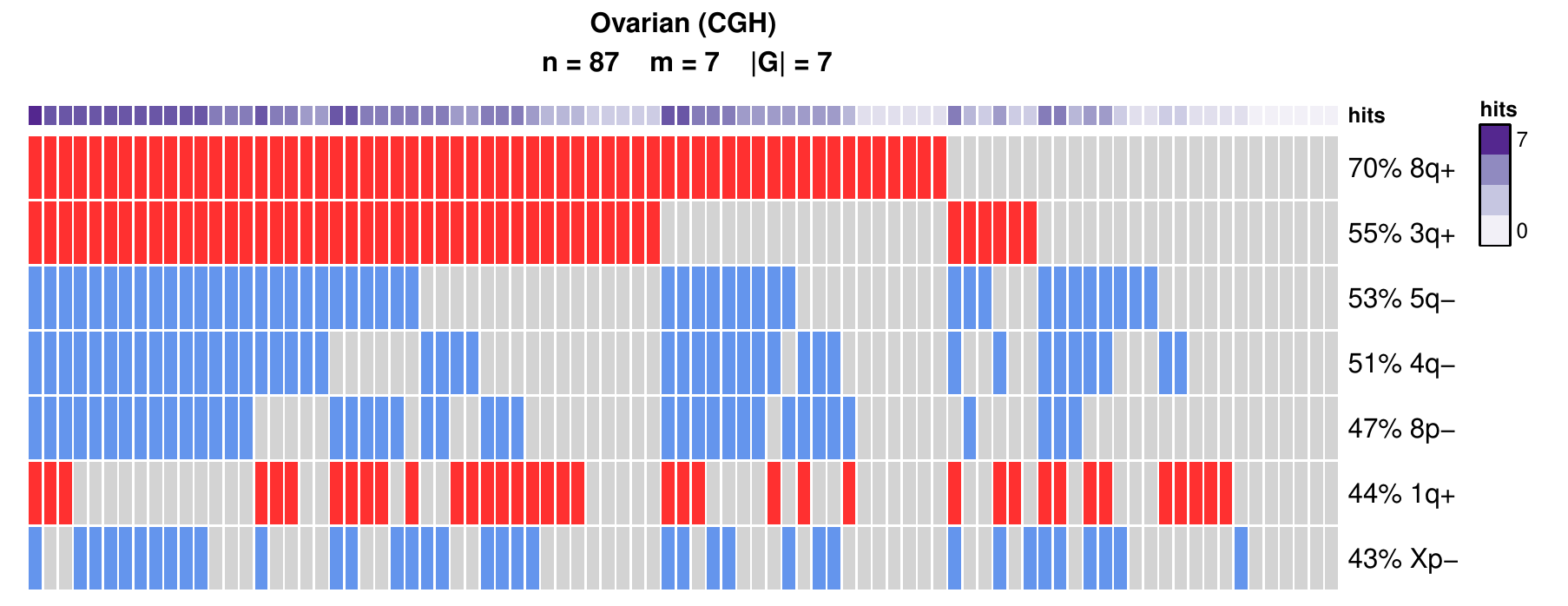} }
\centerline{\includegraphics[width=1.0\textwidth]{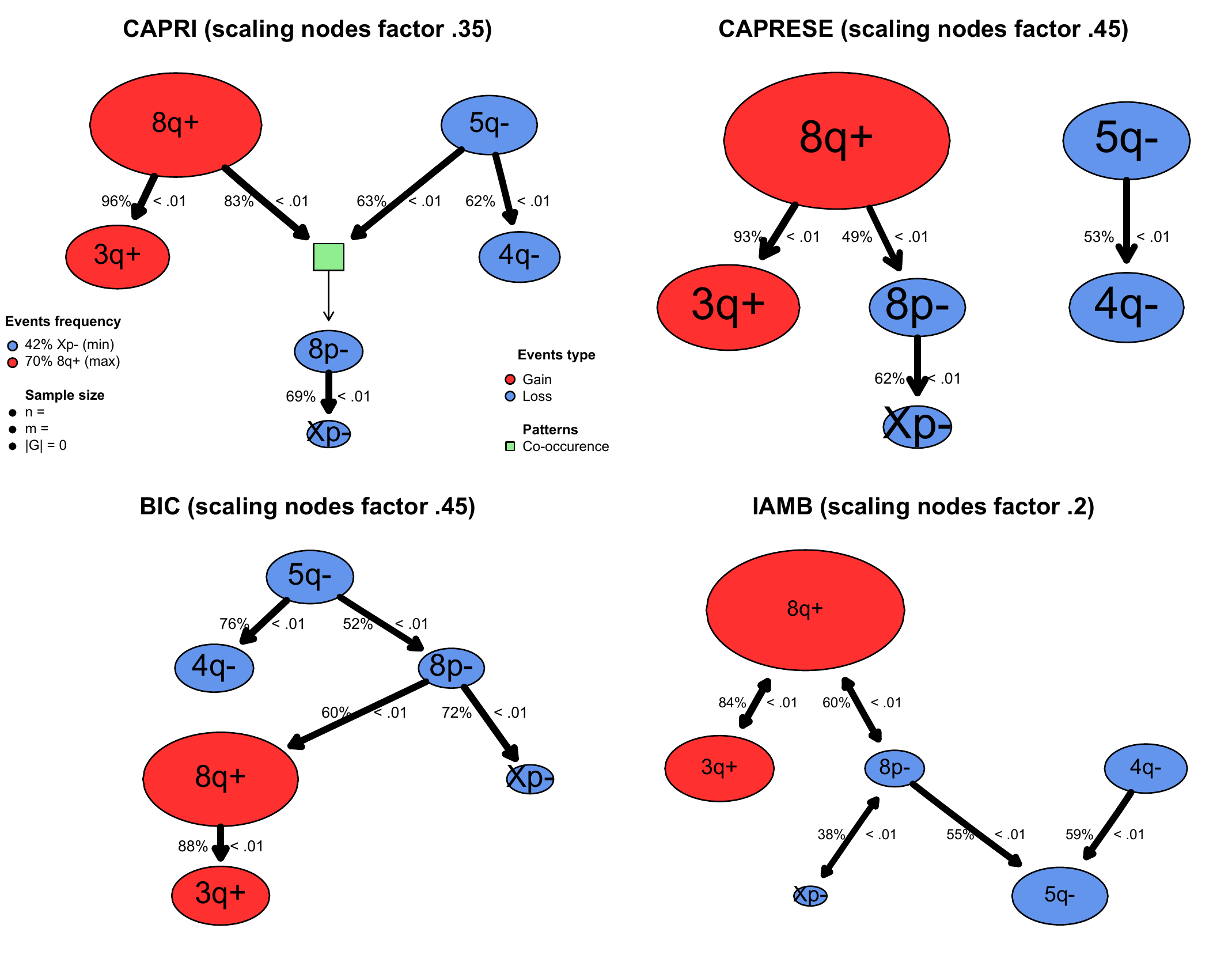} }
\end{center}
\caption[CAPRI, CAPRESE, IAMB and BIC progression models of ovarian cancer]{{\bf CAPRI, CAPRESE, IAMB and BIC progression models of ovarian cancer.} Progression models reconstructed from the ovarian cancer Comparative Genome Hybridization dataset shown in top \cite{knutsen2005interactive}. Algorithms used to infer the models are  CAPRI, CAPRESE, IAMB and BIC. Confidence  is shown as non-parametric bootstrap and hypergeometric test (p-values). Nodes are scaled differently to better layout the graphs reconstructed by every algorithm.}
\label{fig:results_real_ovarian}
\end{figure}
\thispagestyle{empty}
\restoregeometry

\end{document}